\documentclass[a4paper, 11pt]{article}
\usepackage{geometry}
\usepackage[dvipdfmx]{graphicx}

\usepackage{amsmath}						
\usepackage{amsfonts}
\usepackage{latexsym}
\usepackage{amssymb}
\usepackage{bm}
\usepackage{ulem}
\usepackage{cases}
\usepackage{wrapfig,float}

\allowdisplaybreaks[3]

\usepackage{amsthm}							
\usepackage{ascmac}
\usepackage{mathtools}

\usepackage{enumerate}

\theoremstyle{definition}
\newtheorem{theorem}{Theorem}[section]
\newtheorem{prop}[theorem]{Proposition}
\newtheorem{lemma}[theorem]{Lemma}
\newtheorem{cor}[theorem]{Corollary}
\newtheorem{remark}[theorem]{Remark}
\newtheorem{definition}[theorem]{Definition}
\newtheorem{example}[theorem]{Example}
\newtheorem{setting}[theorem]{Setting}

\newtheorem{ques}[theorem]{Question}

\newcounter{res}

\newtheorem{result}[res]{Main Result}   		

\DeclareMathOperator*{\rbif}{\mathit{r}_{\mathrm bif}}
\DeclareMathOperator*{\CC}{\mathbb{C}}
\DeclareMathOperator*{\RR}{\mathbb{R}}
\DeclareMathOperator*{\NN}{\mathbb{N}}
\DeclareMathOperator*{\PP}{\mathbb{P}}
\DeclareMathOperator*{\rs}{\widehat{\mathbb{C}}}
\DeclareMathOperator*{\Mandel}{\mathcal{M}}
\newcommand{\B}[2]{\bar{B}(#1, #2)}
\newcommand{\iteration}[3]{{#1}_{#2}^{\circ #3}}
\newcommand{\random}[3]{f_{#1} \circ \dots \circ f_{#2} \circ f_{#3}}
\newcommand{\sample}[2]{{f}_{#1}^{(#2)}}
\DeclareMathOperator*{\Poly}{\mathrm{Poly}}

\DeclareMathOperator{\supp}{supp}
\DeclareMathOperator{\dist}{dist}

\begin{document}


\title{On the stochastic bifurcations regarding random iterations of polynomials of the form $z^{2} + c_{n}$}

\author{Takayuki Watanabe
\footnote{Chubu University Academy of Emerging Sciences/Center for Mathematical Science and Artificial Intelligence. 
1200 Matsumotocho, Kasugai City, Aichi Prefecture 487-8501, Japan. 
E-mail: takawatanabe@isc.chubu.ac.jp
}}


\maketitle


\begin{abstract}
In this paper, we consider random iterations of polynomial maps $z^{2} + c_{n}$ where $c_{n}$ are complex-valued independent random variables 
following the uniform distribution on the closed disk with center $c$ and radius $r$. 
The aim of this paper is twofold. 
First, we study the (dis)connectedness of random Julia sets. 
Here, we reveal  
the relationships between the bifurcation radius and connectedness of random Julia sets. 
Second, we investigate the bifurcation of our random iterations and give quantitative estimates of bifurcation parameters. 
In particular, we prove that for the central parameter $c = -1$, 
almost every random Julia set is totally disconnected with much smaller radial parameters $r$ than expected. 
We also introduce several open questions worth discussing.  
\end{abstract}






\section[Introduction]{Introduction}
\subsection{Background}
In this paper, we discuss the concrete and interesting random holomorphic dynamical systems with parameters $c \in \CC$ and $r \geq 0$. 
More precisely, we will denote $f_{c}(z) = z^{2} + c$ throughout this paper and 
we consider random iterations of the form $\random{c_{n}}{c_{2}}{c_{1}},$ 
where $c_{n} \ (n=1, \ 2,\ \dots)$ are complex-valued independent random variables 
that follow the uniform distribution on the closed disk $\B{c}{r}=\{c' \in \mathbb{C} \colon |c' - c| \leq r \}$ on the $c$-plane. 
 The reader is referred to Remark 4.10 of \cite{SW22}.  
See Setting \ref{settingUniform} for the rigorous setting. 

The aim of this paper is twofold. 
First, we study the (dis)connectedness of random Julia sets and relate it to the bifurcation, and second, 
 we investigate the bifurcation of our proposed random iterations and give quantitative estimates of bifurcation parameters.

In recent decades, there has been a rapid growth of 
 studies on random dynamical systems. 
In deterministic autonomous dynamical systems, time-evolution rules are independent of time and homogeneous; however, 
they are dependent in random dynamical systems. 
Motivated by scientific demands, great efforts have been  devoted to establishing  fundamental results. 
For this field, refer to the textbook by Arnold \cite{Arn}. 

Forn\ae ss and Sibony first studied random dynamical systems of complex analytic maps \cite{FS91}. 
They proved that small ``random perturbations'' of a single map produce very stable dynamics on average. 
It is then natural to ask what happens when larger noise is added.  
This question is the central motivation of this paper. 

Sumi expanded the results of Forn\ae ss and Sibony by studying minimal sets and established foundations for the study of random holomorphic dynamical systems. 
He defined the mean stability  and showed that the set of mean stable systems is open and dense in the space of all random dynamical systems having mild conditions \cite{Sumi11}. 
We will quickly review this theory in Subsection \ref{ssec:MS}. 
For details, see \cite{Sumi13, Sumi21}. 
See also \cite{SW19, SW22} for the author generalizing the setting from i.i.d.\ to non-i.i.d. setting.  

These previous studies motivate us to study the bifurcation of (i.i.d.) random quadratic dynamical systems induced by the uniform distribution on  $\B{c}{r}$. 
More precisely, for a fixed parameter $c \in \CC$ we can define the bifurcation radius $\rbif(c)$ as the parameter 
satisfying that the random dynamical system is mean stable if and only if $r \notin \{0, \rbif(c)\}$.  
Compare this with  Theorem \ref{th:listOfMinimalSets} and Definition \ref{def:bifRad}.

The dynamical systems of quadratic polynomials $f_{c}(z) = z^{2} + c \ (c \in \mathbb{C})$ are extremely important for studying holomorphic dynamical systems.
The celebrated Mandelbrot set  is an iconic figure on the parameter plane, whose boundary is known as the deterministic bifurcation locus of quadratic polynomials. 
The Mandelbrot set $\mathcal{M}$ is the set of all parameters whose critical orbit is bounded: 
 $\mathcal{M} = \{ c \in \mathbb{C} \colon f_{c}^{\circ n}(0)  \not\to \infty (n \to \infty) \},$ 
 where $f_{c}^{\circ n} = \random{c}{c}{c}$ is the (autonomous) $n$-th iterate of a map $f_{c}$. 
 The condition $ f_{c}^{\circ n}(0)  \not\to \infty$ is equivalent to the Julia set of $f_{c}$ being connected. 
See McMullen's book \cite{McM94} and Milnor's book \cite{Mil06} for details. 

Several authors have reported great studies on sample-wise dynamics for random iterations.  
For a fixed infinite sequence $(c_{n})_{n=1}^{\infty}$, we consider the Julia set of compositions $\random{c_{n}}{c_{2}}{c_{1}}$, 
which is  called a non-autonomous Julia set or a random Julia set. 

Br\"uck, B\"uger, and Reitz investigated the (dis)connectedness of random Julia sets \cite{bbr}. 
They found interesting examples that illustrated the difficulty of random iteration compared with deterministic iteration and 
established several conditions relating the noise amplitude to the connectedness of random Julia sets. 
See Remark \ref{rem:connectedDifficult} and \cite{bb03, b98}.

Br\"uck, B\"uger, and Reitz showed that if center $c =0$ and radius $r \leq 1/4$, 
then every random Julia set of $\random{c_{n}}{c_{2}}{c_{1}}$ is connected. 
Later in \cite{LZ22}, Lech and Zdunik proved that if $c=0$ and $r >1/4$, 
then almost every random Julia set is {\it totally} disconnected. 
For our viewpoint, this is equivalent to $\rbif(0) =1/4.$ 
One motivation of this paper is to generalize this result for the cases where $c \neq 0$. 

\subsection{Main results}
First, we consider a sufficient condition for 
the random Julia sets to be totally disconnected almost surely. 

\begin{result}[Theorem \ref{th:T0=1ImpliesTypFastEsc}]\label{mr:tot}
In this staement, the distribution $\mu$ of $c_{n}$ needs not to be the uniform distribution on a disk. 
Under some assumption, 
if the critical point $0$ tends to $\infty$ with probability one,  
then the random Julia set $J_{\omega}$ is totally disconnected almost surely. 
\end{result}

Second, we reveal  
the relationships between the bifurcation radius $\rbif$ and the connectedness of random Julia sets. 
More precisely, the following holds. 

\begin{result}[Theorem \ref{th:BififfTotDisconn}]\label{mr:A}
Suppose that the interior of $ \B{c}{r}$ contains a superattracting parameter.  
Then the following four are equivalent.
\begin{enumerate}[\hspace{12pt}(1)]
\item The inequality $r \leq \rbif (c)$  holds. 
\item The random orbit of $z=0$ does not tend to $\infty$ surely. 
\item The random Julia set is connected for every sample-path. 
\item The set of all  total orbits of $z=0$ is bounded in $\CC$.   
\end{enumerate}
Also, the following four are equivalent. 
\begin{enumerate}[\hspace{12pt}$(1')$]
\item The inequality $r > \rbif (c)$ holds. 
\item The random orbit of $z=0$ tends to $\infty$ almost surely.  
\item The random Julia set is totally disconnected almost surely. 
\item The set of all  total orbits of $z=0$ is not bounded in $\CC$. 
\end{enumerate}
Furthermore, either the former or the latter is valid. 
\end{result}

In particular, the following dichotomy holds. 

\begin{cor}\label{cor:MR}
Let $c$ be a superattracting parameter. 
If $r \leq \rbif (c)$, then every random Julia set is connected. 
If $r > \rbif (c)$, then almost every random Julia set is totally disconnected. 
\end{cor}

Corollary \ref{cor:MR} is a generalization of the fact that the autonomous Julia set of a single map $f_{c}$ is 
either connected or totally disconnected according to $c \in \Mandel$ or not.

Next, we show some properties and present some quantitative estimates for the bifurcation radius as follows. 

\begin{result}[Lemma \ref{lem:trivialEstimate}, 
Theorems \ref{th:1-Lip}, 
\ref{th:positive}, 
\ref{th:bifMainEstimate}, 
\ref{th:lower-1}, 
Example \ref{ex:?}, and 
Example \ref{ex:airplane}]\label{mr:estimates}
All the following hold. 
\begin{enumerate} 
\item For every $c \in \CC$, we have $\rbif(c) \leq \dist(c, \partial \mathcal{M}).$ 
\item For every $c, c' \in \CC,$ we have $|\rbif(c) - \rbif(c')| \leq |c -c'|.$ 
\item If $0 \leq c \leq 1/4$, then  $\rbif(c) = 1/4 - c$. 
\item If $-1/2 \leq c < 0$, then  $\rbif(c) \leq 1/4 -c - c^{2}.$ 
\item We have  $0.0386\cdots \leq \rbif(-1) \leq 0.0399 \cdots$.  
\item For the airplane parameter $\tilde{c_{3}} \approx -1.75487766$, we have  $\rbif(\tilde{c_{3}}) \leq 0.0021.$  
\end{enumerate}
\end{result}

Here, $\dist(c, \partial \mathcal{M})$ denotes the Euclidean distance from $c$ to the boundary of the Mandelbrot set $\partial \Mandel.$ 
Statement (iii) gives the examples of $c$ satisfying equality $\rbif(c) = \dist(c, \partial \mathcal{M}).$ 
However,  statement (iv), (v), and (vi) give the examples of $c$ satisfying strict inequality $\rbif(c) < \dist(c, \partial \mathcal{M})$. 
  
It is worth noting that for $c=-1$ or $\tilde{c_{3}}$, 
statement (v) and (vi) illustrate a large gap $\rbif(c)\ll  \dist(c, \partial \mathcal{M}).$ 
Furthermore, we also have $\rbif(c + \epsilon)\ll  \dist(c + \epsilon, \partial \mathcal{M})$ for a small $\epsilon \in \CC.$  
The proofs of (v) and (vi) rely on parabolic implosion in a stochastic context. 
See Theorem \ref{th:upperBoundWithSupattrAssumption}. 

Combining (iv) with Corollary \ref{cor:MR},  we give a stronger result compared with the result by Lech and Zdunik. 
Compare Corollary \ref{cor:bifMainEstimate} with Theorem \ref{th:LZ21}. 
Moreover, we give a result for the case where $c$ is far from $0$. 

\begin{cor}[Corollaries \ref{cor:bifMainEstimate} and  \ref{cor:TotDisconn}]
Almost every random Julia set is totally disconnected 
\begin{itemize}
\item if the central parameter $-1/4 \leq c < 0$ and the noise amplitude  $r > 1/4 - c - c^{2},$ or 
\item if the central parameter $c=-1 + \epsilon $ and the noise amplitude  $r > 0.0399 \cdots + |\epsilon| $, or 
\item if the central parameter $c = \tilde{c_{3}} + \epsilon$ and the noise amplitude  $r > 0.0021 + |\epsilon|.$ 
\end{itemize}
\end{cor}

This corollary gives new examples 
which satisfy that the random Julia set is totally disconnected almost surely even if $\B{c}{r} \subset \mathrm{int }\Mandel.$ 

Besides, Main Result \ref{mr:A} gives a quasiconformal conjugacy between Julia sets. 

\begin{result}[Theorem \ref{th:qc} and Corollary \ref{cor:qc}]\label{mr:qc}
Suppose that the interior of $ \B{c}{r}$ contains a superattracting parameter $\tilde{c}$ and   
suppose $r < \rbif (c)$.
Then for every $(c_{n})_{n=1}^{\infty}, (c'_{n})_{n=1}^{\infty} \in \B{c}{r}^{\NN}$, 
there exist $K \geq 1$ and a sequence of maps $\{\varphi_{n}\}_{n \geq 0}$
such that 
 $\varphi_{n}$ maps the Julia sets of $(c_{n})_{n=1}^{\infty}$ onto the Julia sets of $(c'_{n})_{n=1}^{\infty}$ $K$-quasiconformally and 
 $\varphi_{n+1}\circ f_{c_{n+1}}=  f_{c'_{n+1}}  \circ \varphi_{n}$ on the iterated Julia set for every $n \geq 0$.
 In particular, there exists a quasiconformal map $\varphi$ 
 which maps the autonomous Julia set $J_{\tilde{c}}$ onto the non-autonomous Julia set. 
  \end{result}

\subsection{Structure of the paper}
In Section \ref{sec:pre}, we discuss the elementary properties in the general form and  define the bifurcation radius. 
Moreover, we quickly review Sumi's theory on minimal sets and mean stability in Subsection \ref{ssec:MS}. 
In Section \ref{sec:TotDisconn}, we establish some tools to decide connectedness and totally disconnectedness. 
These are developed by  Br\" uck, B\"uger, and Reitz \cite{bbr} and Lech and Zdunik \cite{LZ22}. 
Besides, we show Main Result \ref{mr:tot} relating bifurcation and connectedness. 
In Section \ref{sec:MainRes}, we prove Main Results \ref{mr:A}, \ref{mr:estimates}, and  \ref{mr:qc}. 
In Section \ref{sec:Conc},  we discuss some open problems and generalization to other cases.


\section[Preliminaries]{Preliminaries}\label{sec:pre}
In this section, we define the bifurcation radius and discuss the elementary properties required in subsequent sections. 
We consider minimal sets and mean stability, which appears in Subsection \ref{ssec:MS}. 
They describe set-valued dynamics and we can apply them to quenched, or averaged,  dynamics. 
The minimal sets provide the basis for our discussion and define the bifurcation radius, which is the main subject in this paper.

\subsection{Setting and notations}\label{ssec:setting}
In this section, we consider a slightly general setting. 

\begin{setting}\label{setting} 
Let $\mu$ be a Borel probability measure on the parameter plane $\CC$ with compact support. 
Denote the support by $\supp \mu.$    
We define the probability measure $\mathbb{P}_{\mu}$ as the one-sided infinite product of  $\mu$, 
which is  supported on $\Omega_{\mu} = \prod_{n = 1}^{\infty} \supp \mu$ with the Borel $\sigma$-algebra $\mathfrak{B}_{\mu}$. 
\end{setting}

The most simple and important example for $\mu$ is the normalized Lebesgue measure on $\B{c}{r}  =\{c' \in \mathbb{C} \colon |c' - c| \leq r \}$, which means the uniform distribution. 
Another example is the Dirac measure at $c$, which can be regarded as a deterministic dynamical system. 
We analyze these examples in Section \ref{sec:MainRes}. 

Our random dynamical systems are the stochastic process perturbed by i.i.d. noise in a sense. 
One can generalize the setting to a non-i.i.d. setting, which we do not pursue here. 
For details, the reader is referred to \cite{SW19, SW22}. 

Regarding noise, a natural consideration may be  a distribution $\mu$ with unbounded support, 
say the Gaussian distribution. 
However, in unbounded cases, 
there are no planar attractors which we intend to investigate. 
Thus, we assume that $\supp \mu$ is compact. 
See \cite{b98} for the detail, and see also \cite{CSS20} for more wild phenomena for the distributions with unbounded supports.

The following is a well-known fact. 
See, for example,  Section 6 of \cite{Dur}. 

\begin{lemma}
Denote by $\sigma$ the natural left shift, that is,  $\sigma \omega = (c_{2}, c_{3}, \dots ) $ if $\omega = (c_{1}, c_{2}, \dots )$. 
Then $\sigma \colon \Omega_{\mu} \to \Omega_{\mu}$ is a measure-preserving ergodic transformation with respect to $\mathbb{P}_{\mu}$. 
\end{lemma}

We now define the Julia set for every noise realization $\omega$. 

\begin{definition}\label{def:non-autoIter}
For each  $n \in \NN$ and  infinite sequence of parameters  $\omega = (c_{n})_{n =1}^{\infty} \in \CC^{\NN}$, 
we denote by $\sample{\omega}{n} := \random{c_{n}}{c_{2}}{c_{1}}.$ 
\end{definition}

\begin{definition}
For each $\omega = (c_{n})_{n =1}^{\infty} \in \CC^{\NN}$, we define the non-autonomous Julia set or the random Julia set of $\omega$  by 
$$J_{\omega} = \{ z \in \CC \colon \{\sample{\omega}{n}\}_{n=1}^{\infty} \text{ is not normal on any neighborhood of } z\}.$$
\end{definition}

If there exists $c_{0} \in \CC$ such that $c_{n} = c_{0}$ for every $n \in \NN$, 
then for $\omega = (c_{n})_{n =1}^{\infty}$, 
the non-autonomous Julia set $J_{\omega}$ is equal to the usual (autonomous) Julia set of $f_{c_{0}}$. 

Analogously to the usual case, the non-autonomous Julia set is the common boundary of the basin at infinity and the filled Julia set defined as follows. 

\begin{definition}
For each $\omega = (c_{n})_{n =1}^{\infty} \in \CC^{\NN}$, we define the non-autonomous basin at infinity of $\omega$ by 
$A_{\omega} = \{ z \in \rs \colon \sample{\omega}{n} \to \infty   \text{ as } n\to \infty \}.$
We define the non-autonomous filled Julia set of $\omega$  by $K_{\omega} = \rs \setminus A_{\omega}.$ 
\end{definition}

Here, $\rs = \CC \cup \{ \infty \}$ denotes the Riemann sphere, which is homeomorphic to the real two-dimensional sphere. 
We endow $\rs$ with the spherical metric $d.$ 
A polynomial map $f \colon \CC \to \CC$ can be extended analytically to $f \colon \rs \to \rs$ by letting $f(\infty) = \infty.$ 

For our convenience, 
we list the elementary properties of non-autonomous Julia sets, basins at infinity, and filled Julia sets. 
For $R > 0$, we denote by $D_{R}$ the open disk $D_{R} = \{ z \in \CC \colon |z| < R \}$ on the dynamical plane. 

\begin{lemma}\label{lem:JAK}
Let $\mu$ be a Borel probability measure on $\CC$ with compact support. 
Then we have all the following. 
\begin{enumerate}[(1)]
\item There exists $R > 0$ such that 
$\sample{\omega}{n}(\rs \setminus D_{R}) \subset \rs \setminus D_{2 R}$ for every $n \in \NN$ and $\omega \in \Omega_{\mu}$. 
\item  Take $R > 0$ as above. 
The non-autonomous basin at infinity is then the union of the increasing sequence of the open subsets; $A_{\omega} = \bigcup_{n=1}^{\infty} (\sample{\omega}{n})^{-1}(\rs \setminus \overline{D_{R}})$ for every $\omega \in \Omega_{\mu}$. 
Hence, $A_{\omega}$ is an open neighborhood of $\infty$. 
Conversely, $K_{\omega} = \bigcap_{n=1}^{\infty} (\sample{\omega}{n})^{-1}(\overline{D_{R}})$ is a non-empty compact subset for every $\omega \in \Omega_{\mu}$. 
\item For $\omega \in \Omega_{\mu}$, we have $\partial A_{\omega} = J_{\omega} = \partial K_{\omega}$. 
\item For $\omega \in \Omega_{\mu}$, we have $A_{\omega} = f_{c_{1}}^{-1}(A_{\sigma \omega})$, $K_{\omega} = f_{c_{1}}^{-1}(K_{\sigma \omega})$, and $J_{\omega} = f_{c_{1}}^{-1}(J_{\sigma \omega})$. 
\end{enumerate}
\end{lemma} 
 For the proof, see \cite{bbr}. 
 Note that the existence of such $R$ is due to the compactness of $\supp \mu.$  
 
\subsection{Minimal sets,  mean stability, and bifurcation}\label{ssec:MS}
Consider the action of polynomial semigroups whose product is the composition of maps. 
Semigroup actions are related to random dynamical systems as set-valued dynamics. 
Here, we review the theory of polynomial semigroups, 
which enables us to define bifurcations. 

Generally, one can consider rational semigroups, not polynomial. 
However, we focus on polynomial maps 
since we are interested in quadratic polynomials. 
The interested reader is referred to  \cite{HM} and \cite{stank12}. 

We start with the definition of polynomial semigroups. 
Denote by $\Poly$ the space  of all polynomials of degree two or more endowed with the topology of uniform convergence on $\rs$. 
We say that $G$ is a polynomial semigroup if $G$ is a non-empty subset of $\Poly$ closed under mapping composition. 
Let $\Gamma$ be a subset of $\Poly$. 
We say that $G$ is generated by $\Gamma$ if  
for every $g \in G$, there exist $n \in \NN$ and $\gamma_{1}, \gamma_{2}, \dots \gamma_{n} \in \Gamma$ such that 
$g=\gamma_{n}\circ \dots \circ \gamma_{2} \circ \gamma_{1}.$  

\begin{definition}
Let $\mu$ be a Borel probability measure on  $\CC$ with compact support. 
We define the polynomial semigroup $G_{\mu}$ as the semigroup
generated by the compact set 
$\Gamma = \{f_{c} \in \Poly \colon c \in \supp \mu \}$. 
Here $f_{c}(z) = z^{2} + c$ is a quadratic polynomial, as we will assume throughout this paper. 
\end{definition}

Note that the map $\Poly \times \rs \ni (g, z) \mapsto g(z) \in \rs$ is continuous.  
This implies that the map $\Omega \times \rs \ni (\omega, z) \mapsto \sample{\omega}{n}(z) \in \rs$ is also continuous for every $n \in \NN.$

For every polynomial semigroup, we define the Julia set and Fatou set as follows. 
The following proposition connects the Julia set of polynomial semigroup and the random Julia sets. 
The idea of the proof can be found in \cite{Sumi00}. 

\begin{definition}\label{def:JuliaFatouG}
 For a polynomial semigroup $G$, define the Julia set by 
\begin{align*}
J(G) = \{ z \in \rs \colon G \text{ is not normal on any neighborhood of  } z \}.
\end{align*}
We call the complement $F(G)=\rs \setminus J(G)$ the Fatou set of $G$. 
\end{definition}

\begin{prop}\label{prop:J}
For a Borel probability measure $\mu$  on  $\CC$ with compact support,  
we have  
  $J(G_{\mu}) = \overline{\bigcup_{\omega \in \Omega_{\mu}} J_{\omega}}$.  
\end{prop}

Next, we define minimal sets which play a crucial role in studying random holomorphic dynamical systems.

\begin{definition}\label{def:minimal}
Let $G$ be a polynomial semigroup. 
A subset $L' \subset \rs$ is said to be forward invariant under $G$, or forward $G$-invariant,  if 
$g(L')  \subset L'$ for every $g \in G$. 
A subset  $L\subset \rs$ is called a {\it minimal} set of $G$ if 
$L$ is a minimal element of the set of all forward $G$-invariant, non-empty, and compact subsets of $\rs$ with respect to the inclusion relation $\subset$. 

We can easily show that a non-empty compact set $L\subset \rs$ is minimal if and only if for every $z \in L$, we have $\overline{\bigcup_{g \in G}\{g(z)\}} = L$. 
\end{definition}

Trivially, the singleton $\{\infty\}$ is a minimal set of $G$ for every polynomial semigroup $G$. 
Thus, we are interested mainly in planar minimal sets as defined below. 

\begin{definition}\label{def:planer}
A subset $L \subset \rs$ is said to be {\it planar} if $\infty \notin L$. 
\end{definition}

For example, if $\mu = \delta_{c}$ is the Dirac measure at $c$, 
then every periodic cycle of $f_{c}$ is a minimal set of $G_{\delta_{c}}$. 
Hence, there is an infinite number of planar minimal sets of $G_{\delta_{c}}$ for every $c \in \CC$. 

For the number of minimal sets, we have the following property. 

\begin{lemma}\label{lem:MinimalInclusion}
Let $G_{1}$ and $G_{2}$ be polynomial semigroups generated by $\Gamma_{1}$ and $\Gamma_{2}$,  respectively. 
If $\Gamma_{1} \subset \Gamma_{2}$, 
then every minimal set $L_{2}$ of $G_{2}$ contains some minimal set $L_{1}$ of $G_{1}$. 
\end{lemma}

\begin{proof}
By definition, a minimal set $L_{2}$ of $G_{2}$ is forward invariant under $G_{2}$. 
Since $\Gamma_{1} \subset \Gamma_{2}$, the set $L_{2}$ is  forward invariant also under $G_{1}$. 
It follows from Zorn's lemma that there exists a minimal set $L_{1}$ of $G_{1}$ such that $L_{1} \subset L_{2}$. 
\end{proof}


\begin{definition}\label{def:attractingMinSet}
Let $G$ be the polynomial semigroup generated by a subset $\Gamma$ of $\Poly$. 
A minimal set $L \subset \rs$ of $G$ is said to be {\it attracting} if 
there exist $N \in \NN$ and non-empty open subsets $U$ and $V$ of the Fatou set $F(G)$ 
such that the following two conditions hold. 

\begin{enumerate}
\item $L \Subset V \Subset U \Subset F(G)$. 
\item For every $f_{1}, f_{2}, \dots, f_{N} \in \Gamma$, we have $\random{N}{2}{1}(U) \Subset V.$ 
\end{enumerate}
Here,  $A \Subset B$ means that the closure $\overline{A}$ is a relatively compact subset of $B$.
\end{definition}

One can easily show that if the polynomial semigroup $G$ is generated by a compact subset, 
then the minimal set $\{\infty\}$ is attracting. 
By using the hyperbolic metric, we have the following characterization of attracting minimal sets. 
For the proof, see \cite[Remark 3.5]{Sumi13} or \cite[Lemma 2.8]{SW22}. 

\begin{prop}
Let $G$ be the polynomial semigroup generated by a compact subset $\Gamma$ of $\Poly$. 
A minimal set $L \subset F(G)$ of $G$ is attracting if and only if the following holds. 
There exist $ N \in \NN$, $0 < \alpha < 1$, $C > 0$, and a neighborhood $W$ of $L$ such that 
for every $k \in \mathbb{N}$, $f_{1}, f_{2}, \dots f_{k N} \in \Gamma$, and $z , z_{0} \in W$, the spherical metric satisfies 
$$d (f_{{kN}}\circ \dots \circ f_{{1}} (z),  f_{{kN}}\circ \dots \circ f_{{1}} (z_{0}) ) \leq C {\alpha}^{k} d(z, z_{0}).$$  
\end{prop}

Now, we  present the classification theorem, which states that every minimal set either 
is attracting, intersects the Julia set, or intersects some rotation domain. 
Compare the following theorem with the classification theorem for autonomous iteration of a single map. 
For the latter classical theorem, see \cite[Theorem 16.1]{Mil06} or \cite[Theorem IV.2.1]{CG}. 

\begin{theorem}\label{th:class}
Let $G$ be the polynomial semigroup generated by a compact subset $\Gamma$ of $\Poly$. 
A minimal set $L \subset \rs$ of $G$ is either one of the following three. 
\begin{enumerate}
\item $L$ is attracting. 
\item $L \cap J(G) \neq \emptyset$. 
\item $L \subset F(G)$ and there exist $g \in G$ and a connected component $U$ of $F(G)$ with $L \cap U  \neq \emptyset$ 
such that $g(U) \subset U$ and $U$ is a subset of a Siegel disk of autonomous iteration of $g$. 
\end{enumerate}
\end{theorem}

For the proof, see \cite[Lemma 3.8]{Sumi13}. 
The classification theorem leads to the following definitions. 
We say that a minimal set is {\it J-touching} if it is of type (ii) in Theorem \ref{th:class}. 
Besides, a minimal set is {\it sub-rotative} if it is of type (iii) in Theorem \ref{th:class}. 

However, these two types of minimal sets are uncommon compared with the attracting minimal sets 
as stated by Theorem \ref{th:listOfMinimalSets}.
Denote by $\mathcal{H}$ the set of all parameters $c \in \Mandel$ such that $f_{c}$ is hyperbolic. 

\begin{theorem}\label{th:listOfMinimalSets}
For every $c \in \CC$ and $r > 0$, define the polynomial semigroup $G_{c, r}$ 
as the semigroup generated by $\{f_{c'} \in \Poly \colon |c' - c| \leq r\}$. 
Here $f_{c}(z) = z^{2} + c$. 
Then $G_{c, r}$ has an attracting minimal set $\{\infty\}$. 
Other than $\{\infty\}$, possible planar minimal sets are described as follows. 
\begin{itemize}
\item If $c\not \in \mathcal{H}$, then $G_{c, r}$ has no planar minimal sets for every $r > 0$. 
\item If $c\in \mathcal{H}$, then there exists  $\rbif (c) > 0$ such that the following holds. 
\begin{itemize}
\item If $0 < r < \rbif (c)$, then $G_{c, r}$ has exactly one planar minimal set and it is attracting. 
\item If $r = \rbif (c)$, then $G_{c, r}$ has exactly one planar minimal set and it is not attracting. 
\item If $\rbif (c) < r$, then $G_{c, r}$ has no planar minimal sets. 
\end{itemize}
\end{itemize}
\end{theorem}


Theorem \ref{th:listOfMinimalSets} motivates us to define the bifurcation radius as follows. 
Note that our bifurcation occurs at most once because $f_{c}$ is degree $2,$ 
and in general, our bifurcation occurs at most $d-1$ times if we work on polynomial maps of degree $d$. 

\begin{definition}\label{def:bifRad}
For every $c \in \CC$, we define the  bifurcation radius $\rbif (c) \geq 0$ according to the following cases.  
If $c \notin \mathcal{H}$, we define $\rbif (c) = 0$. 
If $c \in \mathcal{H}$, we define the  bifurcation radius $\rbif (c) > 0$ 
as the unique $r>0$ such that $G_{c, r}$ has the non-attracting planar minimal set.

Equivalently,  the  bifurcation radius $\rbif (c)$ is 
the maximum value for $r \geq 0$ such that  $G_{c, r}$ has a planar minimal set, or equivalently, 
is the infimum value for $r \geq 0$ such that  $G_{c, r}$ has no planar minimal sets. 
\end{definition}

The rest of this subsection is devoted to emphasizing the importance of deciding the bifurcation parameters 
by presenting some interesting phenomena. 
See \cite[Subsection 3.6]{Sumi11} for the details.

\begin{definition}\label{def:MS}
Let $G$ be the polynomial semigroup generated by a compact subset $\Gamma$ of $\Poly$. 
We say that $G$ is mean stable if all the minimal sets of $G$ are attracting. 
\end{definition}

By Theorem \ref{th:listOfMinimalSets}, 
for every $c \in \CC$ and $r \geq 0,$ 
the polynomial semigroup $G_{c, r}$ is mean stable if and only if $r \notin \{0, \rbif(c)\}.$  

Before describing the dynamics of a mean stable system, 
we define the function representing the probability of random orbits tending to an  attractor. 

\begin{definition}\label{def:T}
Let $\mu$ be a Borel probability measure on $\CC$ with compact support. 
For every attracting minimal set $L$ of  $G_{\mu}$, 
we define  the function ${T}_{L, \mu} \colon \rs \to [0, 1]$ by
$${T}_{L, \mu}(z) := \mathbb{P}_{\mu} (\{ \omega \in \Omega_{\mu} \colon d(\sample{\omega}{n} (z), L) \to 0 \, (n \to \infty) \} )$$ 
for every point  $z \in \rs $, 
where $d(w, L) = \min_{w' \in L} d(w, w')$. 
\end{definition} 

The following theorem reveals an interesting noise-induced order. 
That is,  almost every orbit is attracted by some attracting minimal set and 
the probability $T_{L, \mu}(z)$ tending to the attractor $L$ depends continuously on initial point $z$  
if the system is mean stable. 
For the detail,  see \cite[Theorem 3.15]{Sumi11}. 

\begin{definition}
For a polynomial semigroup $G$, define the kernel Julia set by 
\begin{align*}
J_\mathrm{ker}(G) = \bigcap_{g \in G} g^{-1}(J(G)).
\end{align*}
\end{definition}

\begin{theorem}\label{th:Sumi}
Let $\mu$ be a Borel probability measure on $\CC$ with compact support. 
If $J_\mathrm{ker}(G) = \emptyset$, then $T_{\{\infty\}, \mu}$ is continuous on $\rs.$ 
Moreover, 
the sum satisfies $\sum_{L} {T}_{L, \mu}(z) = 1$ for every $z \in \rs$ where the sum runs over all attracting minimal sets  $L$. 
\end{theorem}

The following gives a useful sufficient condition for $J_\mathrm{ker}(G) = \emptyset$. 

\begin{lemma}\label{lem:kernelEmpty}
Let $\mu$ be a Borel probability measure on $\CC$ with compact support. 
If $\mathrm{int} (\supp \mu) \neq \emptyset$, then $J_\mathrm{ker}(G_{\mu}) = \emptyset$. 
\end{lemma}

See \cite[Lemma 5.34]{Sumi11} for the proof. 
Notably, the continuity of ${T}_{L, \mu}$ in Theorem \ref{th:Sumi} does not occur in deterministic dynamical systems. 
Namely, if $\mu$ is the Dirac measure at $c$, 
then the function $T_{\{\infty\}, \mu}$ takes $1$ on $A(f_{c})$ and $0$ on $K(f_{c})$, 
where $A(f_{c}) = A_{\omega}$ and $K(f_{c}) = K_{\omega}$ for $\omega = (c, c, \dots)$. 
Thus, $T_{\{\infty\}, \mu}$ is discontinuous on the autonomous Julia set $J(f_{c}) = J_{\omega}$.  

Theorem \ref{th:Sumi} suggests that the dynamics of a mean stable system are similar to those of a hyperbolic map. 
In the latter case, 
every orbit on the Fatou set is contracted to an attracting cycle, while the complementary Julia set has zero areas. 

The set of all hyperbolic rational maps is conjectured to be open and dense in the space of all rational maps. 
Even in the quadratic polynomial case, whether or not $\overline{\mathcal{H}} = \Mandel$ remains unsolved. 
As is well-known, this conjecture can be affirmatively solved if one prove that the boundary $\partial\Mandel$ is locally connected. 
Meanwhile, it is still open whether or not the boundary of the Mandelbrot set has a positive area, 
although Shishikura showed that it has the full Hausdorff dimension. 

Compared with the autonomous case,  
the following results show that we can solve random versions of the abovementioned conjectures. 
For the proof, see \cite{SW22} and \cite{Sumi13}.  

\begin{theorem}\label{th:RDSversion}
For every $c \in \CC$ and $r \geq 0$, define the polynomial semigroup $G_{c, r}$ 
as the semigroup generated by $\{f_{c'} \in \Poly \colon |c' - c| \leq r\}$. 
Then we have all the following statements. 
\begin{enumerate}
\item The set $\{ (c, r) \in \CC \times [0, \infty ) \colon G_{c, r} \text{ is mean stable} \}$ 
is open and dense in $\CC \times [0, \infty )$ with the usual topology. 
\item For all but countably many $r \in  [0, \infty )$, the set $\{c \in \CC \colon G_{c, r} \text{ is not mean stable} \}$ has zero two-dimensional Lebesgue measure. 
\end{enumerate} 
\end{theorem}

We also have the following tameness if noise is added. 

\begin{theorem}\label{th:AlmostSureAreaZero}
Let  $\mu$ be a Borel probability measure  on  $\CC$ with compact support. 
If $J_\mathrm{ker}(G) = \emptyset$,  then the random Julia set $J_{\omega}$ has zero two-dimensional Lebesgue measure for $\PP_{\mu}$-almost every $\omega \in \Omega_{\mu}$. 
\end{theorem}

One may mistakenly consider that Theorem \ref{th:AlmostSureAreaZero} is insignificant  
since it is difficult to construct autonomous  Julia sets with positive areas. 
However, compared with the autonomous case, 
constructing the non-autonomous Julia set of positive areas is not very difficult. 
Jones {\it et al.} found that  there exists $\omega = (c_{n})_{n=1}^{\infty}$ with $c_{n} \in \{0,\ 1/2,\ 1/4\}$ for all $n \in \NN$ such that $J_{\omega}$ has positive areas.  
They use the parabolic fixed point $1/2$ of the map $f_{1/4}.$ 
For the proof, see \cite[Section 5.2]{Com01}. 
Developing the idea, Comerford proved the existence of measurable invariant sequences of line fields \cite{Com13}. 

As presented above, the dynamics of mean stable systems are similar to those of hyperbolic polynomial maps. 
However, they cause interesting phenomena that cannot occur in deterministic dynamical systems, 
e.g., continuity of $T_{L, \mu}$, finiteness of the set of all minimal sets, etc. 
This motivates us to determine the bifurcation radius where the mean stability breaks.


\section{Total disconnectedness}\label{sec:TotDisconn}
The systematic study of random iterations of the quadratic polynomial was done by Br\"uck and B\"uger. 
In particular, Br\" uck, B\"uger, and Reitz investigated the (dis)connectedness of the non-autonomous Julia sets in \cite{bbr}.  
Here, we present some conditions for the random Julia set to be connected or totally disconnected. 

The following \cite[Theorem 1.1]{bbr} shows that the critical orbits determine connectedness of Julia sets. 
Compared with the autonomous case, we need to modify the statement suitably.  

\begin{theorem}\label{th:conditionToBeConnected}
Let $\omega = (c_{n})_{n =1}^{\infty} \in \CC^{\NN}$. 
Then the non-autonomous Julia set $J_{\omega}$ is connected if and only if the orbit $\{ \sample{\sigma^{k}\omega}{n}(0) \}_{n=1}^{\infty}$ is bounded in $\CC$ for every $k \in \NN$.  
\end{theorem}


Recall that if $\omega = (c_{n})_{n=1}^{\infty}$, then $\sample{\sigma^{k}\omega}{n} = \random{c_{k+n}}{c_{k+2}}{c_{k+1}}$ by definition. 
We remark on the difficulty of the disconnected case compared with the autonomous case. 

\begin{remark}\label{rem:connectedDifficult}
For an autonomous case, the Julia set of a quadratic polynomial is either connected or totally disconnected  \cite[Problem 9-f]{Mil06}. 
However, for a non-autonomous case, we can construct a Julia set that has more than one but finitely many connected components. 
For instance, if $|c_{1}| > 1$ and $c_{n} =0$ for every $n \geq 2$, 
then $J_{\omega} = f_{c_{1}}^{-1}(J_{\sigma \omega})$ has exactly two connected components where $\omega = (c_{n})_{n=1}^{\infty}$. 
Similarly, we can construct  various non-autonomous Julia sets with finitely many connected components. 
Furthermore, there exists $\omega = (c_{n})_{n=1}^{\infty}$ such that $|c_{n}| < 2$ for every $n \in \NN$ and 
$\sample{\sigma^{k}\omega}{n}(0) \to \infty$ as $n \to \infty$ for {\it every} $k \in \NN$, 
but the filled Julia set $K_{\omega}$ contains a segment (see \cite[Example 4.4]{bbr}). 
Thus, it is interesting to determine when  the non-autonomous Julia set $J_{\omega}$ is connected, has finitely many connected components, is totally disconnected, and so on. 
\end{remark}

For the case when $\mu$ is the normalized Lebesgue measure on the disk $\B{0}{r}$, 
Br\" uck, B\"uger, and Reitz showed that $r = 1/4$ is the critical value 
where connectedness drastically changes as we see below. 
If $r \leq 1/4,$ then  the random Julia set $J_{\omega}$ is connected for every $\omega \in \Omega_{\mu}$ \cite[Remark 1.2]{bbr}.  
Br\" uck, B\"uger, and Reitz then questioned whether the random Julia set $J_{\omega}$ is {\it totally} disconnected almost surely if $r > 1/4$. 

Lech and Zdunik \cite{LZ22} showed that it is true. 
We pay attention to their idea for a while.

\begin{definition}\label{def:escapinftime}
Let $\mu$ be a Borel probability measure on  $\CC$ with compact support. 
Let  $R > 0$ be a large number as in Lemma \ref{lem:JAK} (1). 

For every $\omega \in \Omega_{\mu}$, define the escape time $k(z, \omega)$ of $z$ from $D_{R}$ by 

\begin{align*}
k(z, \omega) = 
\begin{cases}
\min \{n \in \NN \colon | \sample{\omega}{n}(z) | > R \} & \text{if } z \in A_{\omega}, \\ 
\infty  & \text{if } z \in K_{\omega}.
\end{cases}
\end{align*}
We say that the critical point $0$ is {\it typically fast escaping} if 
there exists $\gamma > 0$ such that $\PP_{\mu}(\{\omega \in \Omega_{\mu} \colon k(0, \omega) > k\}) \leq e^{{-\gamma k}}$ for every $k \in \NN$. 
\end{definition}

Note that the escape time is related to the value of the Green function 
$$G_{\omega} (z) = \lim_{n \to \infty} \frac{1}{2^{n}} \log_{+}|\sample{\omega}{n}(z) |,$$ 
where $\log_{+}a = \max\{a, \ 0\}$ for $a \in \RR.$  
By \cite[Proposition 8]{LZ22}, 
there exists $C > 0$ such that for every $\omega \in \Omega_{\mu}$ and $z \in A_{\omega} \cap D_{R}$, 
we have 
$$C^{-1} 2^{- k(z, \omega)} \leq G_{\omega} (z) \leq C 2^{- k(z, \omega)}.$$ 
Hence, typical fast escaping can be defined using the Green functions. 
For details, see \cite{LZ22}. 
For other discussions on the Green function, 
see also \cite{FS91, Jon, Sumi06}. 

The following theorem gives a sufficient condition of random Julia sets to be totally disconnected almost surely. 
For the proof, see  \cite[Theorem 11]{LZ22}. 

\begin{theorem}\label{th:TypFastEscImpliesTotDisconn}
Let $\mu$ be a Borel probability measure on  $\CC$ with compact support. 
If the critical point $0$ is typically fast escaping, 
then the Julia set $J_{\omega}$ is totally disconnected for $\PP_{\mu}$-almost every $\omega \in \Omega_{\mu}$.
\end{theorem}

Using the theorem above the authors of \cite{LZ22} proved the following theorem. 

\begin{theorem}\label{th:LZ21}
Let $\mu$ be a Borel probability measure on  $\CC$ with compact support. 
If $\supp \mu \supset \B{0}{1/4}$ and $\supp \mu \neq \B{0}{1/4}$, 
then the random Julia set $J_{\omega}$ is totally disconnected for $\PP_{\mu}$-almost every $\omega \in \Omega_{\mu}$. 
\end{theorem}

Consider the case where $\mu$ is the normalized Lebesgue measure on the disk $\B{c}{r}$. 
From above, the connectedness changes at $r = 1/4$ if $c = 0$. 
Since the value $1/4$ is the distance from $c = 0$ to the boundary of the Mandelbrot set, 
one may conjecture that the transition from connected to disconnected Julia sets equals the distance from $c$ to the boundary of the Mandelbrot set for any $c \in \Mandel$. 
However, Theorem \ref{th:LZ21} states that even if $\supp \mu$ is contained in the closure of the main cardioid of the Mandelbrot set, it can happen that the random Julia set $J_{\omega}$ is totally disconnected almost surely. 

Before giving the sufficient condition of typical fast escaping, we show the following lemma. 

\begin{lemma}\label{lem:EIsAlmostSurelyinvariant}
Let $\mu$ be a Borel probability measure on  $\CC$ with compact support. 
Define $E = \{z \in \rs \colon {T}_{\{\infty\}, \mu}(z) = 1\}$.  
If $z \in E$ and $k \in \NN$, then $\sample{\omega}{k}(z) \in E$ for $\PP_{\mu}$-almost every $\omega \in \Omega_{\mu}.$ 
\end{lemma}

\begin{proof}
Since $\PP_{\mu}$ is the product of the probability measure $\mu$, 
Fubini's theorem implies $$T_{\{\infty\}, \mu}(z) = \int_{\B{c}{r}} T_{\{\infty\}, \mu}(f_{c_{1}} (z)) \mathrm{d} \mu (c_{1})$$ 
for every $z \in \rs.$ 
By induction, we have 
$T_{\{\infty\}, \mu}(z) = \int_{\Omega} T_{\{\infty\}, \mu}(\sample{\omega}{k} (z)) \mathrm{d}\PP_{\mu} (\omega).$ 
Thus, if $T_{\{\infty\}, \mu}(z) = 1$, then $T_{\{\infty\}, \mu}(\sample{\omega}{k} (z)) = 1$ with probability one. 
\end{proof}

The following is the new and original result that gives a sufficient condition for typical fast escaping. 

\begin{theorem}\label{th:T0=1ImpliesTypFastEsc}
Let $\mu$ be a Borel probability measure on  $\CC$ with compact support. 
Suppose that $T_{\{\infty\}, \mu}$ is continuous. 
If $T_{\{\infty\}, \mu}(0) = 1,$   
then the critical point $0$ is typically fast escaping, 
and hence the random Julia set $J_{\omega}$ is totally disconnected for $\PP_{\mu}$-almost every $\omega \in \Omega_{\mu}$. 
\end{theorem}

\begin{proof}
The latter statement is a consequence of Theorem \ref{th:TypFastEscImpliesTotDisconn}. 
Thus, it suffices to show that the critical point $0$ is typically fast escaping. 
Let  $R > 0$ be a large number as in Lemma \ref{lem:JAK} (1). 
Set $E = \{z \in \rs \colon {T}_{\{\infty\}, \mu}(z) = 1\}$,  
then $E$ is compact since we assume that $T_{\{\infty\}, \mu}$ is continuous. 
(This is the only place where the assumption of continuity is used.) 

For every $z \in E$, there exist $n_{z} \in \NN$ and $c^{z}_{1}, c^{z}_{2}, \dots c^{z}_{n_{z}} \in \supp \mu$ such that 
 $|\random{c^{z}_{n_{z}}}{c^{z}_{2}}{c^{z}_{1}} (z)| > R$. 
 Since this map $\random{c^{z}_{n_{z}}}{c^{z}_{2}}{c^{z}_{1}} $ is continuous, 
 there exists an open neighborhood $O_{z}$ of $z$ such that for every $z' \in O_{z}$, we have 
 $|\random{c^{z}_{n_{z}}}{c^{z}_{2}}{c^{z}_{1}} (z')| > R$. 
We may and do assume that $O_{z}$ is precompact. 

We now consider the open covering $\{O_{z}\}_{z \in E}$ of $E$. 
Since $E$ is compact, there exists a finite subcover. 
Namely, there exist $\ell \in \NN$ and $\ell$ open sets $O_{j}$ $(j =1, \dots, \ell)$ such that 
$E \subset \bigcup_{j=1}^{\ell} O_{j}.$ 
Moreover, for each $j =1, \dots, \ell$, there exist 
 $n_{j} \in \NN$ and $c^{j}_{1}, c^{j}_{2}, \dots c^{j}_{n_{z}} \in \supp \mu$ such that 
 for every $z' \in O_{j}$, we have
 $|\random{c^{j}_{n_{j}}}{c^{j}_{2}}{c^{j}_{1}} (z')| > R$. 
Let $N = \max_{j =1, \dots, \ell} n_{j}$. 
By Lemma \ref{lem:JAK} (1), we can assume that $n_{j} = N$ for every $j=1, \dots, \ell$. 
Since $\CC \ni c \mapsto f_{c} \in \Poly$ is continuous, 
for every $n =1, 2, \dots , N$, there exists an open neighborhood $U^{j}_{n}$ of $c^{j}_{n}$ such that 
for every $c_{n} \in U^{j}_{n}$ and $z' \in O_{j}$, we have  $|\random{c_{N}}{c_{2}}{c_{1}} (z')| > R$. 

Define $\mathfrak{S}^{k}(z) = \{ \omega \in \Omega_{\mu} \colon |\sample{\omega}{k}(z) | \leq R \}$ 
for every $z \in \rs$. 
Then, by the construction above, we have 
$\PP_{\mu} (\Omega_{\mu} \setminus \mathfrak{S}^{N}(z)) \geq \PP_{\mu} (\prod_{n=1}^{N} U^{j}_{n} \times \prod_{N+1}^{\infty} \B{c}{r})$ 
for every $j =1, \dots, \ell$ and  $z \in O_{j}$. 
Since $c^{j}_{n} \in \supp \mu$ and $U^{j}_{n}$ is its open neighborhood, 
we have $\mu(U^{j}_{n}) > 0$ for every $j =1, \dots, \ell$ and $n =1, 2, \dots , N$. 
Define $\alpha = \min_{j =1, \dots, \ell} \prod_{n=1}^{N} \mu(U^{j}_{n}) > 0.$ 
Then, for every $z \in E$, we have $\PP_{\mu} (\mathfrak{S}^{N}(z)) \leq 1-\alpha.$ 

It follows from Fubini's theorem that 
$$\mathbb{P}_{\mu} (\mathfrak{S}^{k + N}(z)) 
= \int_{\Omega_{\mu}} \mathbb{P}_{\mu} (\mathfrak{S}^{N}(\sample{\omega}{k}(z))) \mathrm{d} \mathbb{P}_{\mu}(\omega)$$ for every $k \in \NN.$ 
If $\omega \not\in \mathfrak{S}^{k}(z)$, then $\mathfrak{S}^{N}(\sample{\omega}{k}(z)) = \emptyset$ by Lemma \ref{lem:JAK}. 
Thus, the integral on the complement of $\mathfrak{S}^{k}(z)$ is zero; 
$\PP_{\mu} (\mathfrak{S}^{k + N}(z))  
 =  \int_{\mathfrak{S}^{k}(z)} \PP_{\mu} (\mathfrak{S}^{N}(\sample{\omega}{k}(z))) \mathrm{d}\PP_{\mu}(\omega)$. 
By Lemma \ref{lem:EIsAlmostSurelyinvariant}, 
we have $\PP_{\mu} (\mathfrak{S}^{k + N}(z)) \leq (1 - \alpha) \PP_{\mu} (\mathfrak{S}^{k}(z) ) $ 
if $z \in E$. 

Repeating this, we have 
$\PP_{\mu} ( \mathfrak{S}^{mN}(0)) \leq (1 - \alpha)^{m}  $
for every $m \in \NN.$ 
Thus, we can find a constant $\gamma > 0$ such that 
 $$\mathbb{P}_{\mu}(\{\omega \in \Omega_{\mu} \colon k(0, \omega) > k\})\leq \mathbb{P}_{\mu} ( \mathfrak{S}^{k}(0))  \leq e^{{-\gamma k}}$$ for every  $k \in \NN$. 
 This completes the proof. 
\end{proof}

\begin{cor}\label{cor:T0=1ImpliesTypFastEsc}
Let $\mu$ be a Borel probability measure on  $\CC$ whose support is compact and satisfies $\mathrm{int} (\supp \mu) \neq \emptyset$. 
If $T_{\{\infty\}, \mu}(0) = 1,$   
then the random Julia set $J_{\omega}$ is totally disconnected for $\PP_{\mu}$-almost every $\omega \in \Omega_{\mu}$. 
\end{cor}

\begin{proof}
This is a consequence of Lemma \ref{lem:kernelEmpty} and Theorems \ref{th:Sumi} and \ref{th:T0=1ImpliesTypFastEsc}. 
\end{proof}

\begin{cor}\label{cor:cor}
Let $\mu$ be a Borel probability measure on  $\CC$ whose support is compact. 
If    $G_{\mu}$ has no planar minimal sets, 
then the random Julia set $J_{\omega}$ is totally disconnected for $\PP_{\mu}$-almost every $\omega \in \Omega_{\mu}$. 
\end{cor}

\begin{proof}
Since $G_{\mu}$ has no planar minimal sets, the function satisfies   $T_{\{\infty\}, \mu} \equiv 1$   
by Theorem \ref{th:Sumi}. 
Since the constant function is continuous, Theorem \ref{th:T0=1ImpliesTypFastEsc} implies the conclusion. 
\end{proof}

For instance,  
let $\mu = \delta_{0}/2 + \delta_{c}/2$ be the convex combination of the Dirac measures at $0$ and $c$ with $|c| > 1.$ 
If $|z| > 1$, then $\iteration{f}{0}{n}(z) \to \infty$ as $n \to \infty.$ 
If $|z| \leq 1$, then $|f_{c}(z)| \geq -|z^{2}| + |c| > 1$, and hence $\iteration{f}{0}{n} \circ f_{c}(z) \to \infty$ as $n \to \infty.$ 
This shows that $G_{\mu}$ has no planar minimal sets, 
which implies that the random Julia set is totally disconnected $\PP_{\mu}$-almost surely. 

In the next section, we apply Corollary \ref{cor:T0=1ImpliesTypFastEsc} to the case where $\mu$ is the normalized Lebesgue measure on the disk $\B{c}{r}$.

\section{Main results on uniform noise process}\label{sec:MainRes}
Throughout this section, we consider the case where $\mu$ is the normalized Lebesgue measure on the disk $\B{c}{r}$ as in Setting \ref{settingUniform}.

\begin{setting}\label{settingUniform} 
For fixed $c \in \CC$ and $r \geq 0$, 
we define the following Borel probability measure $\mu_{c, r}$ on $\CC$ with compact support. 
If $r > 0$, define $\mu_{c, r}$ as the normalized Lebesgue measure on $\B{c}{r} =\{c' \in \mathbb{C} \colon |c' - c| \leq r \}$. 
If $r = 0$, define $\mu_{c, r}$ as the Dirac measure at $c$. 
We define the probability space $(\Omega_{c, r}, \ \PP_{c, r}) = (\Omega_{\mu_{c, r}}, \ \PP_{\mu_{c, r}})$ as in Setting \ref{setting}. 
Namely, denote by $\mathbb{P}_{c, r} $  the one-sided infinite product of  $\mu_{c, r}$ 
supported on $\Omega_{c, r} = \prod_{n = 1}^{\infty} \B{c}{r}$ with the Borel $\sigma$-algebra. 
Besides, let the polynomial semigroup $G_{c, r}$ be defined  
by the semigroup generated by $\{f_{c'} \in \Poly \colon c' \in \B{c}{r}\}$ 
and define the function $T_{c, r} \colon \rs \to [0, 1]$, abusing notation, by $$T_{c, r}(z) = \mathbb{P}_{c, r} (\{ \omega \in \Omega_{c, r} \colon \sample{\omega}{n} (z) \to \infty \, (n \to \infty) \} ).$$ 
\end{setting}

\begin{remark}
Setting \ref{settingUniform} can be interpreted as an additive-noise stochastic process in the following way. 
Fix $c \in \CC$ and $r \geq 0$. 
On a probability space $(\Omega,\ \mathfrak{B}, \ \PP)$, take independent random variables $\nu_{n} \colon \Omega \to \CC$ that follow the uniform distribution on $\B{0}{r}$ for all $n \in \NN$. 
Fix  $z_{0} \in \CC$ and define the random variable $Z_{0} \colon \Omega \to \CC$ by a constant map with value $z_{0}$. 
Define the $\CC$-valued stochastic process $\{Z_{n}\}_{n \geq 0}$ on $(\Omega,\ \mathfrak{B}, \ \PP)$,  
inductively by $Z_{n+1} =f_{c}(Z_{n}) + \nu_{n+1}$ for every $n \geq 0$. 
Then the distribution of $Z_{n}$ are identical to 
that of the random variable $\sample{\bullet}{n} (z_{0}) \colon \Omega_{c,r} \to \CC$ for every $n \in \NN$, 
where $\sample{\bullet}{n} (z_{0})$ is the map $\omega \mapsto \sample{\omega}{n} (z_{0}).$ 
This stochastic process describes the process driven by adding uniform noise independently to the deterministic dynamical systems of $f_{c}$ with the initial value $z_{0}$. 
Here, the parameter $r$ describes the size of noise that satisfies $| \nu_{n} | \leq r$ almost surely for every $n \in \NN$. 
\end{remark}

We investigate the bifurcation radius $\rbif(c)$ defined in Definition \ref{def:bifRad}. 
Our final goal is to determine $\rbif(c)$ and we provide some partial results in this paper. 

\subsection{Relation between bifurcation and connectedness}
In this subsection, we reveal the relation between bifurcation and connectedness. 
We start the discussion with the following corollary. 

\begin{cor}\label{cor:randomJuliaSetIsTotDisconn}
Suppose that $r > \rbif(c)$. 
Then the random Julia set $J_{\omega}$ is totally disconnected for $\PP_{c, r}$-almost every $\omega$. 
\end{cor}

\begin{proof}
If $r > \rbif(c),$ then  $T_{c, r} \equiv 1$  
by Theorems \ref{th:listOfMinimalSets} and \ref{th:Sumi}. 
It follows from Corollary \ref{cor:T0=1ImpliesTypFastEsc}  
that the random Julia set $J_{\omega}$ is totally disconnected almost surely. 
\end{proof}

In the following, suppose that the interior of the support $\B{c}{r}$ contains a superattracting parameter. 

\begin{lemma}\label{lem:cond}
Suppose that the interior of $ \B{c}{r}$ contains a superattracting parameter $\tilde{c}$. 
If there exists $\omega \in \Omega_{c, r}$ such that $\sample{\omega}{n} (0) \to \infty$ as $n \to \infty$, 
then  $r > r_\mathrm{bif}(c)$. 
\end{lemma}

\begin{proof}
We show that for every $z \in \rs$, there exist $c_{1}, c_{2} , \dots \in \B{c}{r}$ such that $f_{c_{n}} \circ \dots  \circ f_{c_{2}} \circ f_{c_{1}} (z) \to \infty$, 
which implies that $G_{c, r}$ has no planar minimal sets. 

The proof is divided into three cases. 
Denote by $K(f_{\tilde{c}})$ the autonomous filled Julia set of the superattracting map $f_{\tilde{c}}.$ 
If $z \not\in K(f_{\tilde{c}})$, then  $ f_{\tilde{c}}^{\circ n} (z) \to \infty$ by definition.  
If $z \in \partial K(f_{\tilde{c}})$, then there exists $c_{1} \in \B{c}{r}$ such that $f_{c_{1}}(z) \not\in K(f_{\tilde{c}})$ and hence $ f_{\tilde{c}}^{\circ n} \circ f_{c_{1}} (z) \to \infty$. 
 If $z \in \mathrm{int }K(f_{\tilde{c}})$, then by the No Wandering Domain Theorem, 
the orbit $ f_{\tilde{c}}^{\circ n} (z) $ converges to the attracting cycle of $0$.  
By the assumption, we have $0 \in A_{\omega}.$ 
Since the non-autonomous basin $A_{\omega}$ is open,
 there exists an open neighborhood $D$ of $0$ such that for every $z' \in D$, 
 we have $\sample{\omega}{n} (z') \to \infty$ as $n \to \infty$.  
Take $N \in \NN$ so that $ f_{\tilde{c}}^{\circ N} (z) \in D$, then 
 $\sample{\omega}{n} ( f_{\tilde{c}}^{\circ N} (z)) \to \infty$ as $n \to \infty$. 
 This completes the proof. 
\end{proof}

\begin{lemma}\label{lem:boundedattractor}
Suppose that the interior of $ \B{c}{r}$ contains a superattracting parameter $\tilde{c}$.  
If $L = \overline{\bigcup_{g \in G_{c, r}}\{g(0)\}}$ is bounded in $\CC$, then $L$ is a planar minimal set of $G_{c, r}$. 
\end{lemma}

\begin{proof}
Fix any $z \in L$. 
First, we show $\overline{\bigcup_{g \in G_{c, r}}\{g(z)\}} \subset L.$ 
By definition, there exists $g_{n} \in G_{c, r}$ such that $g_{n}(0) \to z$ as $n \to \infty.$ 
Then, for every $g \in G_{c, r}$, we have  $g \circ g_{n}(0) \to g(z)$ as $n \to \infty.$  
Since $G_{c, r}$ is a semigroup, we have  $g \circ g_{n} \in G_{c, r}$. 
Thus, we have $g(z) \in L,$ which implies   $\overline{\bigcup_{g \in G_{c, r}}\{g(z)\}} \subset L.$ 
Next, we show $\overline{\bigcup_{g \in G_{c, r}}\{g(z)\}} \supset L.$ 
For every open set $U$ which intersects $L$, take $g_{1}  \in G_{c, r}$ such that $g_{1}(0) \in U.$ 
Since $g_{1}$ is continuous, there exists an open set  $U_{0}$ such that $0 \in U_{0}$ and $g_{1}(U_{0}) \subset U.$ 
Since $\overline{\bigcup_{g \in G_{c, r}}\{g(z)\}} \subset L,$ 
the autonomous orbit $\{\iteration{f}{\tilde{c}}{n}(z)\}_{n \in \NN}$ is bounded,  
and hence $z \in K(f_{\tilde{c}})$. 
If $z \in \partial K(f_{\tilde{c}})$, then there exists $c_{1} \in \B{c}{r}$ such that $f_{c_{1}}(z)  \not \in K(f_{\tilde{c}}),$ 
but this contradicts  $\overline{\bigcup_{g \in G_{c, r}}\{g(z)\}} \subset L.$ 
Thus,  $z \in \mathrm{int } K(f_{\tilde{c}})$. 
It follows from the No Wandering Domain Theorem that 
 there exists $N \in \NN$ such that $\iteration{f}{\tilde{c}}{N}  (z) \in U_{0}$. 
Thus, $g_{1} \circ \iteration{f}{\tilde{c}}{N} (z) \in U,$ which implies $\overline{\bigcup_{g \in G_{c, r}}\{g(z)\}} \supset L.$ 
This completes the proof. 
\end{proof}

The following is the main theorem of this paper.

\begin{theorem}\label{th:BififfTotDisconn}
Suppose that the interior of $ \B{c}{r}$ contains a superattracting parameter $\tilde{c}$.  
Then the following four are equivalent.
\begin{enumerate}[\hspace{12pt}(1)]
\item The inequality $r \leq \rbif (c)$  holds. 
\item For every $\omega \in \Omega_{c, r}$, the orbit $\{ \sample{\omega}{n}(0) \}_{n=1}^{\infty}$ is bounded; hence $T_{c, r}(0) = 0$. 
\item The random Julia set $J_{\omega}$ is connected for every $\omega \in \Omega_{c, r}$. 
\item The set $\overline{\bigcup_{g \in G_{c, r}}\{g(0)\}}$ is a planar minimal set of $G_{c, r}$. 
\end{enumerate}
In contraposition to this, the following four are equivalent. 
\begin{enumerate}[\hspace{12pt}$(1')$]
\item The inequality $r > \rbif (c)$ holds. 
\item For $\PP_{c, r}$-almost every $\omega \in \Omega_{c, r}$, the orbit $\sample{\omega}{n}(0) \to \infty \ (n \to \infty)$: $T_{c, r}(0) = 1$. 
\item The random Julia set $J_{\omega}$ is totally disconnected for $\PP_{c, r}$-almost every $\omega \in \Omega_{c, r}$. 
\item The set $\overline{\bigcup_{g \in G_{c, r}}\{g(0)\}}$ is not bounded in $\CC$. 
\end{enumerate}
Furthermore, either the former or the latter is valid. 
\end{theorem}

\begin{proof}
It is obvious that $(1)$ or $(1')$ is valid. 
First, we show the equivalence of the four from $(1)$ to $(4).$ 

 $(1)\Rightarrow (2):$  
If $r \leq \rbif (c)$, then $\sample{\omega}{n} (0) \not\to \infty$ as $n \to \infty$ for every $\omega \in \Omega_{c, r}$ by Lemma \ref{lem:cond}. 
Hence,   $T_{c, r}(0) = 0$. 

 $(2)\Rightarrow (3):$  
Take a large $R > 0$ as in Lemma \ref{lem:JAK} (1). 
 Suppose that for every $\omega \in \Omega_{c, r}$, we have  $\{\sample{\omega}{n} (0) \}_{n=1}^{\infty}$ is bounded. 
 Then $|\sample{\omega}{n} (0)| < R$ for every $n\in \NN$. 
By Theorem \ref{th:conditionToBeConnected},  the random Julia set $J_{\omega}$ is connected for every $\omega \in \Omega_{c, r}$. 
 
$(3)\Rightarrow (4):$  
Suppose that the random Julia set $J_{\omega}$ is connected for every $\omega \in \Omega_{c, r}$. 
For a large $R > 0$ as in Lemma \ref{lem:JAK} (1), 
it follows from Theorem \ref{th:conditionToBeConnected} that 
$|\sample{\sigma^{k}\omega}{n}(0)|  \leq R$  for every $\omega \in \Omega_{c, r}$,  $n \in \NN$, and $k \in \NN$.  
Thus, the set $L = \overline{\bigcup_{g \in G_{c, r}}\{g(0)\}}$ is bounded in $\CC.$ 
By Lemma \ref{lem:boundedattractor}, the set $L$ is a planar minimal set.

$(4)\Rightarrow (1):$  
This is trivial. 
See the latter part of Definition \ref{def:bifRad}. 

The argument above completes the proof of the equivalence from $(1)$ to $(4).$ 
Next, we show the equivalence of the four from $(1')$ to $(4').$ 

 $(1')\Rightarrow (2'):$  
 If $r > \rbif (c)$, then $G_{c, r}$ has no planar minimal sets by Theorem \ref{th:listOfMinimalSets}. 
 It follows from Theorem \ref{th:Sumi} that $T_{c, r}(z) = 1$ for every $z \in \rs$, and hence $T_{c, r}(0) = 1$. 

$(2')\Rightarrow (3'):$ 
Suppose $T_{c, r}(0) = 1$. 
If $r>0$, 
then the random Julia set $J_{\omega}$ is totally disconnected $\PP_{c, r}$-almost surely by Corollary \ref{cor:T0=1ImpliesTypFastEsc}. 
If $r = 0,$ 
then it trivially follows that the autonomous Julia set is totally disconnected. 

$(3')\Rightarrow (4'):$ 
If $(3')$ holds, then there exists $\omega \in \Omega_{c, r}$ such that the random Julia set $J_{\omega}$ is disconnected.
By Theorem \ref{th:conditionToBeConnected}, there exists $k \in \NN$ such that $\sample{\sigma^{k}\omega}{n}(0) \to \infty$ as $n \to \infty.$ 
Since $\sample{\sigma^{k}\omega}{n} \in G_{c, r}$, then $\overline{\bigcup_{g \in G_{c, r}}\{g(0)\}}$ is unbounded. 

$(4')\Rightarrow (1'):$ 
Let $R > 0$ be large, as in Lemma \ref{lem:JAK} (1). 
If  $\overline{\bigcup_{g \in G_{c, r}}\{g(0)\}}$ is not bounded, then there exists $g \in G_{c, r}$ such that $|g(0)| > R.$ 
By definition of $G_{c, r},$  there exist $N \in \NN$ and $c_{1}, c_{2}, \dots, c_{N} \in \B{c}{r}$ 
such that $g = \random{c_{N}}{c_{2}}{c_{1}}.$ 
Letting $c_{n} = c$ for every $n \geq N+1$ and $\omega = (c_{n})_{n=1}^{\infty}$, we have $\sample{\omega}{n}(0) \to \infty$ as $n \to \infty$ by Lemma    \ref{lem:JAK}. 
It follows from Lemma \ref{lem:cond} that $r > \rbif (c)$. 
\end{proof}

\begin{remark}
 The orbit $\{ \sample{\omega}{n}(0) \}_{n=1}^{\infty}$ is bounded for every $\omega \in \Omega_{c, r}$
 if and only if 
 $T_{c, r}(0) = 0$
 \end{remark}

\begin{proof}
It is obvious that ``only if'' part is true. 
 Suppose $T_{c, r}(0) = 0$ and we show there does not exist $\omega \in \Omega_{c, r}$ such that $\sample{\omega}{n} (0) \to \infty$ as $n \to \infty$ by contradiction. 
 Take a large number $R > 0$ as in Lemma \ref{lem:JAK} (1). 
 If there exists $\omega = (c_{n})_{n \in \NN} \in \Omega_{c, r}$ such that $\sample{\omega}{n} (0) \to \infty$, 
 then there exists $N \in \NN$ such that $|\random{c_{N}}{c_{2}}{c_{1}} (0)| > R.$ 
Since $\CC \ni c \mapsto f_{c} \in \Poly$ is continuous, 
there exists an open neighborhood $U_{n}$ of $c_{n}$ for every $n =1, 2, \dots, N$ such that 
$|\random{c'_{N}}{c'_{2}}{c'_{1}} (0)| > R$ for every $c'_{n} \in U_{n}$ ($n =1, 2, \dots, N$). 
Moreover, we have $\PP_{c, r} ( \mathfrak{U} ) > 0$ for $ \mathfrak{U} = \prod_{n=1}^{N} U^{j}_{n} \times \prod_{N+1}^{\infty} \B{c}{r}.$ 
However, for every $\omega' \in \mathfrak{U}$, it follows from Lemma \ref{lem:JAK} that $\sample{\omega'}{n} (0) \to \infty$ as $n \to \infty$, 
contradicting the assumption $T_{c, r}(0) = 0$. 
This shows that for every $\omega \in \Omega_{c, r}$, we have  $|\sample{\omega}{n} (0)| \leq R$ for every $n \in \NN.$ 
 \end{proof}

\subsection{General properties}\label{ssec:general}
In this subsection, we give a trivial estimate of the bifurcation radius and show the continuity. 
The following is due to \cite{bbr}.


\begin{lemma}\label{lem:MandelInside}
If $\B{c}{r} \cap (\mathbb{C}\setminus \Mandel) \neq \emptyset$, 
then for every $z \in \CC$, there exists $\omega \in \Omega_{c, r}$ such that 
$\sample{\omega}{n} (z) \to \infty$ as $n \to \infty.$ 
\end{lemma}

\begin{proof}
Fix $c' \in \B{c}{r} \setminus \Mandel.$ 
Then the autonomous filled Julia set $K(f_{c'})$ of $f_{c'}$ has no interior point. 
Thus, if $z \in K(f_{c'})$, then there exists $c_{1} \in \B{c}{r}$ such that $f_{c_{1}}(z) \notin K(f_{c'})$. 
In this case, we have $f_{c'}^{\circ n} \circ f_{c_{1}}(z) \to \infty$ as $n \to \infty.$ 
Also, if $z \notin K(f_{c'})$, then by definition $f_{c'}^{\circ n}(z) \to \infty$ as $n \to \infty.$ 
This completes the proof.
\end{proof}

Denote by $\dist(c, \partial \mathcal{M})$ the Euclidean distance from the point $c$ to the compact set $\partial \mathcal{M}$.  

\begin{lemma}\label{lem:trivialEstimate}
For every $c \in \CC$, we have $\rbif(c) \leq \dist(c, \partial \mathcal{M}).$ 
\end{lemma}

\begin{proof}
From Lemma \ref{lem:MandelInside},  if $r  >  \dist(c, \partial \mathcal{M}),$ 
then the polynomial semigroup $G_{c, r}$ has no planar minimal sets. 
By Definition \ref{def:bifRad}, $\rbif(c)$ is the infimum value for $r \geq 0$ such that  $G_{c, r}$ has no planar minimal sets. 
Hence, $\rbif(c) \leq  \dist (c, \partial \mathcal{M}). $
\end{proof}

The following is a part of Main Result \ref{mr:estimates}. 

\begin{theorem}\label{th:1-Lip}
As a function, 
the bifurcation radius $\rbif \colon \mathbb{C} \to [0, \infty)$ is $1$-Lipschitz. 
That is,  
$|\rbif(c) - \rbif(c')| \leq |c -c'|$ for every $c, c' \in \CC.$  
\end{theorem}

\begin{proof}
Fix $c, c' \in \CC$. 
Pick any $r > | c -c' | $ and set $\epsilon = r - | c -c' | > 0$. 
By definition, the polynomial semigroup $G_{c, \rbif(c)+\epsilon}$ has no planar minimal sets. 
Since  $\B{c}{\rbif(c) + \epsilon} \subset  \B{c'}{\rbif(c)+r}$, 
it follows from Lemma \ref{lem:MinimalInclusion}  that also  $G_{c', \rbif(c)+r}$ has no planar minimal sets. 
This implies that $\rbif(c') \leq \rbif(c)+r$, and hence $\rbif(c') - \rbif(c) \leq  | c -c' |.$ 
Exchanging $c$ and $c'$ each other, we have also $\rbif(c) - \rbif(c') \leq  | c' -c |,$ which completes the proof.  
\end{proof}

\subsection{Inside the main cardioid}
In this subsection, we present some estimates of $\rbif(c)$ when $c$ is in the main cardioid. 
Recall that the main cardioid is the set of all parameters $c \in \CC$ for which there exists $\lambda \in \CC$ with $|\lambda| < 1$ such that $c =  \lambda/2 -\lambda^{2}/4,$ 
and its boundary is contained in $\partial \mathcal{M}.$

\begin{cor}\label{cor:rbif0}
If $ c = 0$, then  $\rbif(0) = 1/4$. 
\end{cor}

\begin{proof}
This is essentially due to \cite{bbr}. 
By Lemma \ref{lem:trivialEstimate}, we have $\rbif(0) \leq 1/4$. 
A straight calculation shows that $\overline{D_{1/2}}$ is forward invariant under the polynomial semigroup $G_{0, 1/4}$, and hence $\rbif(0) \geq 1/4.$ 
\end{proof}

We can easily verify that $\overline{D_{1/2}}$ is the planar minimal set of $G_{0, 1/4}$. 
Since the autonomous Julia of $f_{1/4}$ contains $z = 1/2$ which is the parabolic fixed point of $f_{1/4},$ 
the Julia set of the semigroup satisfies $1/2 \in J(G_{0,1/4})$ by Proposition \ref{prop:J}.
Thus, the planar minimal set $\overline{D_{1/2}}$ of $G_{0, 1/4}$ is J-touching at $z=1/2$. 

Corollary \ref{cor:rbif0} shows that the equality in Lemma \ref{lem:trivialEstimate} may hold.  
Now, we show using Theorem \ref{th:1-Lip} that we can determine the bifurcation radius if $0 \leq c \leq 1/4$. 
The following is a part of Main Result \ref{mr:estimates}. 

\begin{theorem}\label{th:positive}
If $0 \leq c \leq 1/4$, then  $\rbif(c) = 1/4 - c$. 
\end{theorem}

\begin{proof}
By Lemma \ref{lem:trivialEstimate}, we have $\rbif(c) \leq 1/4 - c$ if $0 \leq c \leq 1/4$. 
If $c' = 0$, then $\rbif(c') = 1/4$ from Corollary \ref{cor:rbif0}. 
Thus, Theorem \ref{th:1-Lip} yields that $|\rbif(c) - 1/4| \leq c$ for $0 \leq c \leq 1/4$, which completes the proof.  
\end{proof}

We showed that $\rbif(c) = \dist(c, \partial \mathcal{M})$ if $0 \leq c \leq 1/4$. 
However, the author does not know whether there exists $c \in \CC$ other than these, such that $0 < \rbif(c) = \dist(c, \partial \mathcal{M}).$
For $c < 0$, 
we will present some examples which satisfy $\rbif(c) < \dist(c, \partial \mathcal{M})$.  

\begin{lemma}\label{lem:maincadioid}
For $c = 0$ and $0 < r_{0} \leq 1/4$, let $\delta = (1 - \sqrt{1-4 r_{0}})/2$. Then the following holds.  
\begin{enumerate}
\item  The equality $\delta^{2} + r_{0} = \delta$ holds. Thus, the point $\delta$ is a fixed point of $f_{r_{0}}$. 
\item The set $\overline{D_{\delta}}$ is the planar minimal set of $G_{0, r_{0}}.$ 
\end{enumerate}
\end{lemma}

\begin{proof}
Statement (i) follows from the direct calculation.  
Now, we look at statement (ii). 
By (i), we have $f_{c_{1}}(\overline{D_{\delta}}) \subset \overline{D_{\delta}}$  for every $c_{1} \in \B{0}{r_{0}}$. 
Since $\B{0}{r_{0}}$ is symmetric under rotation, 
then the union $\bigcup_{c_{1} \in \B{0}{r_{0}}} \{f_{c_{1}}(0) \}$ is a round disk centered at the origin. 
By induction, $\bigcup_{c_{1}, c_{2}, \dots, c_{n} \in \B{0}{r_{0}}} \{\random{c_{n}}{c_{2}}{c_{1}}(0) \}$ is a round disk centered at the origin for every $n \in \NN$, and also $\overline{\bigcup_{g \in G_{0, r_{0}}}\{g(0)\}}$ is a closed round disk contained in $\overline{D_{\delta}}$. 
If $c_{n}=r_{0}$ for every $n \in \NN$, then the autonomous iteration $\iteration{f}{r_{0}}{n}(0)$ converges to the fixed point $\delta$ of $f_{r_{0}}$. 
This yields that $\overline{\bigcup_{g \in G_{0, r_{0}}}\{g(0)\}} = \overline{D_{\delta}}$, and hence $\overline{D_{\delta}}$ is a minimal set of $G_{0, r_{0}}$. 
\end{proof}

The following is a part of Main Result \ref{mr:estimates}. 

\begin{theorem}\label{th:bifMainEstimate}
Suppose that $0 < \epsilon \leq 1/2.$ 
For  $c=- \epsilon,$ we have $\rbif(-\epsilon) \leq 1/4 + \epsilon - \epsilon ^{2}.$ 
\end{theorem}

\begin{proof}
Fix any $r > 1/4 + \epsilon - \epsilon^{2}.$ 
We show that $G_{-{\epsilon}, r}$ has no planar minimal sets by contradiction, which implies $\rbif(-\epsilon) < r$. 
Suppose that there exists a planar minimal set $L$ of $G_{-{\epsilon}, r}$. 
Since $\B{-\epsilon}{r} \supset \B{0}{r- \epsilon}$,  
it follows from  Lemma \ref{lem:MinimalInclusion} that $G_{0, r-\epsilon}$ also has a planar minimal set. 
By Corollary \ref{cor:rbif0} we have $ r- \epsilon \leq 1/4$.
Besides, we assumed that $r  -  \epsilon> 1/4 - \epsilon^{2} \geq 0.$ 
By Lemmas \ref{lem:MinimalInclusion} and  \ref{lem:maincadioid} (ii), 
  we have $L \supset \overline{D_{\delta}}$ for $\delta = (1 - \sqrt{1-4 (r - \epsilon) })/2$. 
Let $z_{0}= i \delta$ and $c_{1} = -r - \epsilon$, then $z_{0} \in L$, $c_{1} \in \B{-\epsilon}{r}$, and 
$f_{c_{1}}(z_{0}) = - \delta^{2} - r - \epsilon = - \delta - 2 \epsilon.$  
Now, set $R_{1} = (1 + \sqrt{1 - 4 (r - \epsilon) })/2 > 1/{2}$, 
then 
\begin{align*}
- f_{c_{1}}(z_{0})  -R_{1} 
&= 2 \epsilon + \frac{1 - \sqrt{1-4 (r - \epsilon) }}{2} - \frac{1 + \sqrt{1-4 (r - \epsilon) }}{2} 
=2 \epsilon - \sqrt{1-4 (r - \epsilon)} > 0
\end{align*}
since $r > 1/4 + \epsilon - \epsilon^{2}$. 
Set $z_{1} = f_{c_{1}}(z_{0})$, then  $z_{1} < - R_{1}$. 
Let  $c_{n}=r  - \epsilon \in \B{-\epsilon}{r}$ for every $n \geq 2$. 
Since $z_{1} < - R_{1}$, the direct calculation shows $f_{c_{2}}(z_{1}) > R_{1}^{2} + r - \epsilon = R_{1}.$ 
Set $z_{2} = f_{c_{2}}(z_{1})$ and $\alpha = z_{2} - R_{1}$, then  $z_{2} = R_{1} + \alpha$ and $\alpha > 0$. 
We can show by induction that  $\iteration{f}{r - \epsilon }{n}(z_{2}) > R_{1} + \alpha + n \alpha^{2}$ for every $n \in \NN$.  
Hence, $\iteration{f}{r - \epsilon }{n}(z_{1})$ diverges to $\infty$ as $n \to \infty$. 
Thus, we have proved that there exist $z_{0} \in L$ and $\omega = (c_{n})_{n=1}^{\infty} \in \Omega_{-\epsilon, r}$ such that $\sample{\omega}{n} (z_{0}) \to \infty$ as $n \to \infty.$ 
This contradicts that $L$ is planar and forward invariant under $G_{-{\epsilon}, r}$. 
This completes the proof since $r$ is an arbitrary number that satisfies $r > 1/4 + \epsilon - \epsilon^{2}.$ 
\end{proof}

Note that we can show that $\dist ( -\epsilon, \partial \Mandel ) = 1/4 +\epsilon$ if  $0 < \epsilon \leq 1/4$ as follows. 
Assume $0 < \epsilon \leq 1/4$, 
then $\dist ( -\epsilon, \partial \Mandel ) \leq |-\epsilon - 1/4| = 1/4 +\epsilon$ since $1/4 \in \Mandel$. 
Let $C = \{ \lambda/2 -\lambda^{2}/4 \colon \lambda = e^{i \theta}, \theta \in \RR \}$.  
Then $C \subset \partial \Mandel$ and the domain bounded by $C$ is contained in the interior of $\Mandel$, see \cite[Theorem VIII.1.3]{CG}. Since 
\begin{align*}
\left|-\frac{1}{4} - \frac{e^{i \theta}}{2} + \frac{e^{2i \theta}}{4}\right|^{2}
&=
\left(-\frac{1}{4} - \frac{e^{i \theta}}{2} + \frac{e^{2i \theta}}{4}\right)
\left(-\frac{1}{4} - \frac{e^{-i \theta}}{2} + \frac{e^{-2i \theta}}{4}\right)\\
&= \frac{6 - (e^{2 i \theta} + e^{-2 i \theta})}{16} = \frac{6 - 2 \cos 2 \theta}{16}, 
\end{align*}
we have $\dist (-1/4, \partial \Mandel) = \min_{\theta \in \RR}\sqrt{6 - 2 \cos 2 \theta}/4 = 1/2$ and the minimum is attained at $\theta = 0, \pi$ which correspond to $c = 1/4, -3/4 \in \Mandel$ respectively.  
Since $\mathrm{int}  \B{-\epsilon}{1/4 +\epsilon} \subset \mathrm{int}  \B{-1/4}{1/2} \subset \mathrm{int} \Mandel$, 
we have $\dist ( -\epsilon, \partial \Mandel ) \geq 1/4 +\epsilon.$ 
Therefore, Theorem \ref{th:bifMainEstimate} reveals that strict inequality $\rbif(c) < \dist(c, \partial \mathcal{M})$ holds if $-1/4 \leq c < 0$. 

Combining the theorem above with Corollary \ref{cor:randomJuliaSetIsTotDisconn}, 
we have a little better result than Theorem \ref{th:LZ21}. 

\begin{cor}\label{cor:bifMainEstimate}
Suppose that $0 < \epsilon \leq 1/2.$ 
If $c=- \epsilon$ and $r > 1/4 + \epsilon - \epsilon ^{2}$, 
then the random Julia set $J_{\omega}$ is totally disconnected for $\PP_{-\epsilon, r}$-almost every $\omega \in \Omega_{-\epsilon, r}$. 
\end{cor}

This shows that 
the random Julia set $J_{\omega}$ is totally disconnected for almost every $\omega$ 
even if $\B{c}{r} \subset \mathrm{int} \Mandel$ and $\B{c}{r} \not \supset \B{0}{1/4 }$. 
This is stronger than Theorem \ref{th:LZ21} proved by Lech and Zdunik.

\subsection{Outside the main cardioid}
Recall that a parameter $\tilde{c}$ is said to be superattracting if there exists $p \in \NN$ such that $\iteration{f}{\tilde{c}}{p}(0) = 0$. 
For example, $\tilde{c} = 0$ and $-1$ are superattracting parameters with $p = 1$ and $2$, respectively. 
In this subsection, we show the nontrivial estimates of $\rbif(-1)$.

\begin{lemma}\label{lem:-1invariant}
Let $c=-1$, $r \geq 0$, and $\delta \geq 0$. 
Denote $\overline{{D}}_{\delta} = \{z \in \mathbb{C} \colon |z| \leq \delta \}$. 
Then 
$f_{c_{2}} \circ f_{c_{1}}(\overline{{D}}_{\delta}) \subset \overline{{D}}_{\delta}$ for every $c_{1}, \ c_{2} \in \B{-1}{r}$ if and only if 

\begin{align}\label{ineq-1}
\delta^{4}+2(1+r)\delta^{2} + r^{2}  + 3r \leq \delta.
\end{align}
\end{lemma}

\begin{proof}
Assume that $f_{c_{2}} \circ f_{c_{1}}(\overline{{D}}_{\delta}) \subset \overline{{D}}_{\delta}$ for every $c_{1}, \ c_{2} \in \B{-1}{r}$. 
For $z = i \delta$, $c_{1} = -1 -r$, and $c_{2} = -1 +r$, 
we have $f_{c_{1}}(z) = - \delta^{2} -1-r$ and $f_{c_{2}} \circ f_{c_{1}}(z)  = \delta^{4} +2(1+r)\delta^{2} + r^{2}  + 3r.$ 
Hence, inequality (\ref{ineq-1}) holds. 

Conversely, assume that inequality (\ref{ineq-1}) holds. 
If $z \in \overline{{D}}_{\delta}$ and $c_{1} \in \B{-1}{r}$, 
then $|f_{c_{1}}(z) +1| \leq |z|^{2} + |c_{1}+1| \leq \delta^{2} + r$. 
Besides, if $|z_{1} + 1| \leq \delta^{2} + r$ and $c_{2} \in \B{-1}{r}$, 
then  $$|f_{c_{2}}(z_{1})| = |(z_{1} + 1)^{2} -2 z_{1} -2 + 1 + c_{2}| \leq |z_{1} +1|^{2} + 2|z_{1} +  1| + | 1 + c_{2} | \leq \delta^{4}+2(1+r)\delta^{2} + r^{2}  + 3r. $$ 
Thus, if inequality (\ref{ineq-1}) holds, then 
$f_{c_{2}} \circ f_{c_{1}}(\overline{{D}}_{\delta}) \subset \overline{{D}}_{\delta}$ for every $c_{1}, \ c_{2} \in \B{-1}{r}$. 
\end{proof}

\begin{lemma}\label{lem:equa-1}
Denote by $\delta_{\mathrm{max}} = 0.453\cdots$ the (positive real) root of $\delta^{3} + 2\delta-1$. 
Define the function $\rho \colon [0, \delta_{\mathrm{max}}] \to \RR$ by 
$$\rho (\delta) = - \frac{2 \delta^{2} + 3}{2} + \frac{\sqrt{4 \delta^{2} + 4\delta +9}}{2}.$$
Then, for every $\delta \in [0, \delta_{\mathrm{max}}] $, we have $\delta^{4}+2(1+\rho(\delta) )\delta^{2} + \rho(\delta) ^{2}  + 3\rho(\delta)  = \delta$ and $\rho(\delta) \geq 0.$ 
\end{lemma}

\begin{proof}
The former part $\delta^{4}+2(1+\rho(\delta) )\delta^{2} + \rho(\delta) ^{2}  + 3\rho(\delta)  = \delta$ follows from a straight calculation. 
We show $\rho(\delta) \geq 0.$ 
A simple calculation shows that 
$$2 \rho (\delta) = \frac{- ({2 \delta^{2} + 3})^{2} + 4 \delta^{2} + 4\delta +9}{(2 \delta^{2} + 3) + \sqrt{4 \delta^{2} + 4\delta +9}} 
= \frac{- 4 \delta (\delta^{3} + 2\delta-1)}{(2 \delta^{2} + 3) + \sqrt{4 \delta^{2} + 4\delta +9}}.$$ 
Define $\tilde{\rho}(\delta) = \delta^{3} + 2\delta-1,$ then $\tilde{\rho}(0) = -1 <0$ and $\tilde{\rho}'(\delta) = 3 \delta^{2} + 2 > 0$. 
Thus, $\rho(\delta) \geq 0$  for every $\delta \in [0, \delta_{\mathrm{max}}] $. 
\end{proof}

We now give the lower bound of $\rbif(-1),$ 
which is one of the main results of this paper.  

\begin{theorem}\label{th:lower-1}
Let $\rho$ be defined as in Lemma \ref{lem:equa-1}. 
Then $\rho$ takes the maximum value $r_{\mathrm{max}} = 0.0386\cdots$, and
hence  $r_{\mathrm{max}} \leq r_{\mathrm{bif}}(-1)$.  
\end{theorem}

\begin{proof}
We have that the derivative $\rho'(\delta) = -2 \delta +(2 \delta + 1)/\sqrt{4 \delta^{2} + 4 \delta + 9}.$ 
Thus, $\rho'(\delta) = 0$ if and only if $16 \delta^4 + 16 \delta^3 + 32\delta^2 - 4 \delta - 1 = 0$. 
Let $\delta^{*} =  0.229\cdots $ be the unique positive root of this quartic equation. 
Then, $\rho(\delta)$ takes the maximum value at $\delta = \delta^{*}$. 
A numerical computation shows that  the maximum value $r_{\mathrm{max}}$ is approximately $0.0386\cdots.$ 

By Lemmas \ref{lem:-1invariant} and \ref{lem:equa-1},  
for every $\delta \in [0, \delta_{\mathrm{max}}] $, 
the set $\overline{{D}}_{\delta} $ is forward invariant under the polynomial semigroup $G_{-1, \rho(\delta)}$. 
Thus, we have $\rho(\delta) \leq r_{\mathrm{bif}}(-1)$.  
In particular, $r_{\mathrm{max}} = \rho(\delta^{*}) \leq r_{\mathrm{bif}}(-1)$.  
This completes the proof. 
\end{proof}

Next, we consider the upper bound of $\rbif(-1).$ 
The following theorem gives upper bounds, 
whose proof suggests a kind of parabolic implosion in the stochastic sense. 

\begin{theorem}\label{th:upperBoundWithSupattrAssumption}
Suppose that  $ \B{c}{r_{0}}$ contains a superattracting parameter $\tilde{c}$. 
If there exist $N \in \NN$ and $c_{1}, c_{2} \dots, c_{N} \in \B{c}{r_{0}}$ such that $f_{c_{N}} \circ \dots  \circ f_{c_{2}} \circ f_{c_{1}}$ has a parabolic periodic point, 
then  $r_\mathrm{bif}(c) \leq r_{0}$. 
\end{theorem}

\begin{proof}
Fix $r > r_{0}$ and apply Lemma \ref{lem:cond}. 
Set $g = f_{c_{N}} \circ \dots  \circ f_{c_{2}} \circ f_{c_{1}}$ and  
let $z^{*}$ denote the parabolic periodic point of $g$.  
Then there exists  $p\in \NN$ such that$ \iteration{g}{}{p}(z^{*}) = z^{*}$, and the multiplier of $\iteration{g}{}{p}$ at  $z^{*}$ is $1.$ 
Thus, there exists a critical point $z_{0}$ of $g$ such that $\iteration{g}{}{np}(z_{0}) \to z^{*}$ as $n \to \infty.$ 
Since $g = f_{c_{N}} \circ \dots  \circ f_{c_{2}} \circ f_{c_{1}}$, there exists $0 \leq k \leq N-1$ such that $f_{c_{k}} \circ \dots  \circ f_{c_{2}} \circ f_{c_{1}} (z_{0}) = 0$ by the chain rule. 
Here, if $k =0 $, then we define $f_{c_{k}} \circ \dots  \circ f_{c_{2}} \circ f_{c_{1}}$ to be the identity map. 
In any case, we have $\iteration{g}{}{np-1} \circ \random{c_{N}}{c_{k+2}}{c_{k+1}}(0) \to z^{*}$ as $n \to \infty.$ 
Now, let $\epsilon = r - r_{0} > 0$, then there exists $n_{0} \in \NN$ such that $n_{0} \geq 2$ and 
$$|\iteration{g}{}{n_{0}p-1} \circ \random{c_{N}}{c_{k+2}}{c_{k+1}}(0) - z^{*} | < \epsilon.$$ 
Set $w = \iteration{g}{}{n_{0}p-1} \circ \random{c_{N}}{c_{k+2}}{c_{k+1}}(0)$.  
Since the parabolic periodic point $z^{*}$ belongs to the autonomous Julia set $J(g)$ of $g$ and the Julia set is the boundary of the basin at infinity, 
perturbation of $c_{N}$ allows us to find a random orbit which escapes to the basin at infinity. 
Namely, there exists $c' \in \B{c_{N}}{\epsilon}$ such that $$w' = f_{c'} \circ f_{c_{N-1}} \circ \dots  \circ f_{c_{2}} \circ f_{c_{1}} \circ \iteration{g}{}{n_{0}p-2} \circ \random{c_{N}}{c_{k+2}}{c_{k+1}}(0) \in A(g),$$ 
where $A(g)$ denotes the basin at infinity for the autonomous iteration of $g$. 
Note that $c' \in \B{c}{r}$ since $c_{N} \in \B{c}{r_{0}}$ and $\epsilon = r - r_{0}$. 
By definition, 
we have $\iteration{g}{}{n}(w') \to \infty$ as $n \to \infty$. 
Thus,  Lemma \ref{lem:cond} gives that $\rbif(c) < r.$ 
This completes the proof since $r$ is an arbitrary number that satisfies $r >r_{0}.$  
\end{proof}

Theorem \ref{th:upperBoundWithSupattrAssumption} gives an alternative proof of $\rbif(0) \leq 1/4$ 
since $f_{1/4}$ has a parabolic fixed point at $z =1/2$. 

The following gives an algorithm to find an upper bound of the bifurcation radius. 

\begin{theorem}\label{th:algorithm}
Suppose $c$ is a superattracting parameter. 
Fix $N \in \NN$ and $k \in \NN$, and   
let $p_{i} \in \{0, 1, \dots, k-1\} $ for each $i =1, 2, \dots, N.$ 
Define $\zeta_{k} = e^{2 \pi i/k}$ and  $c_{i} = c + \rho \zeta_{k}^{p_{i}},$ 
where $\rho$ is a complex variable.  
Let $\Delta (\rho)$ be the discriminant of the algebraic equation $f_{c_{N}} \circ \dots  \circ f_{c_{2}} \circ f_{c_{1}}(z) - z =0$ in the variable $z$, which is a polynomial in the variable $\rho$. 
If $\Delta(\rho_{0}) = 0$ for some $\rho_{0},$ then $r_\mathrm{bif}(c) \leq |\rho_{0}|.$ 
\end{theorem}

\begin{proof}
Under the assumptions, if $\rho = \rho_{0},$ then 
$f_{c_{N}} \circ \dots  \circ f_{c_{2}} \circ f_{c_{1}}(z) - z =0$ has a multiple root. 
This implies that the map $f_{c_{N}} \circ \dots  \circ f_{c_{2}} \circ f_{c_{1}}$ has a parabolic fixed point with multiplier $1.$ 
Since $c_{1}, c_{2} \dots, c_{N} \in \B{c}{r_{0}}$ with $r_{0} = |\rho_{0}|,$ we have $r_\mathrm{bif}(c) \leq r_{0}$ by Theorem \ref{th:upperBoundWithSupattrAssumption}. 
\end{proof}

We now give upper bounds of $\rbif(-1)$. 


\begin{example}\label{ex:?}
Suppose $c=-1.$ 
Let $N = 4$, $k =6,$ and $(p_{1}, p_{2}, p_{3}, p_{4}) = (1, 2, 5,4 ).$ 
Then numerical calculation shows the following result. 
The discriminant $\Delta(\rho)$ of $f_{c_{4}} \circ f_{c_{3}}  \circ f_{c_{2}} \circ f_{c_{1}}(z) -z $ 
vanishes at  $\rho_{0} \approx 0.0399 \cdots$, 
and hence $\rbif(-1) \leq 0.0399 \cdots$. 
\end{example}

We have some comments on the upper bounds of $\rbif(-1).$ 

\begin{remark}\label{rem:NumComp}
In  Example \ref{ex:?}, the numerical errors should be treated carefully 
because 
the discriminant of the equation $f_{c_{4}} \circ f_{c_{3}}  \circ f_{c_{2}} \circ f_{c_{1}}(z) - z =0$ with respect to $z$ 
is a polynomial $\Delta({\rho})$ of degree $32$, and 
each coefficient is a very large number in absolute value, i.e., 
about $10^{19}$-$10^{26}$ order.  
The author used validated numerics to estimate $\rho_{0}$, which showed $\rbif(-1) \leq |\rho_{0}| \leq 0.0399217$. 
Since validated numerics gives values including mathematically strict error evaluation, 
this value $0.0399217$ is a rigorous upper bound of $\rbif(-1).$ 
For more details on validated numerics, see \cite{PP} for example. 

Note also that  all the $6^{4}$ candidates 
$(p_{1}, p_{2}, p_{3}, p_{4}) \in  \{0,  1, 2, 3, 4, 5\}^{4}$ were (non-rigorously) verified. 
That is, for every $(p_{1}, p_{2}, p_{3}, p_{4}) \in  \{0,  1, 2, 3, 4, 5\}^{4}$, 
the discriminant $\Delta({\rho})$ of the equation $f_{c_{4}} \circ f_{c_{3}}  \circ f_{c_{2}} \circ f_{c_{1}}(z) - z =0$ was computed using $c_{1}, \dots, c_{4}$ as in Theorem \ref{th:algorithm}, 
and we numerically computed the roots of $\Delta({\rho})$. 
This computation shows that $24$ candidates $(p_{1}, p_{2}, p_{3}, p_{4})$ including $(1, 2, 5,4 )$ attain the same minimum absolute value $|\rho_{0}| \approx 0.0399 \cdots$. 
Also, the author  replaced $\zeta_{6}$ by $\zeta_{4} = e^{2 \pi i/4} = i$, and computed  the roots of $\Delta({\rho})$ in the same manner, 
which gave less sharp estimates. 
\end{remark}

We now give the upper bound for another parameter. 

\begin{example}\label{ex:airplane}
Suppose $\tilde{c_{3}} \approx -1.75487766$ is the airplane parameter
 such that $\tilde{c_{3}}$ is real number and  $f_{\tilde{c_{3}}}$ has a superattracting periodic point with period $3$. 
 Let $N = 3$, $k =6,$ and $(p_{1}, p_{2}, p_{3}) = (0, 3, 0).$ 
Then a similar numerical computation shows that $\rbif(\tilde{c_{3}}) \leq 0.0021.$  
\end{example}

\begin{remark}\label{rem:bif-1issmall}
Recall that we can calculate the connected component $W$ of $\mathrm{int} \mathcal{M}$ which contains $-1$. 
The component $W$ is the open disk with center $-1$ and radius $1/4,$ thus $\dist (-1, \partial \Mandel) = 1/4.$
For $c=-1$, we have $\rbif(-1) \ll \dist (-1, \partial \Mandel)$ by the example above. 
Similarly, $\rbif(\tilde{c_{3}}) \ll \dist (\tilde{c_{3}}, \partial \Mandel)$ since  $ \dist(\tilde{c_{3}}, \partial \Mandel) \approx 0.00487766$ for  the airplane parameter $\tilde{c_{3}}$. 
See \cite{GF95} for the explicit parametrization of period-$3$ hyperbolic components. 
This strict inequality is in contrast with the equality $\rbif(0) = \dist(0, \partial \Mandel)$. 
\end{remark}

\begin{cor}\label{cor:TotDisconn}
Almost every random Julia set is totally disconnected 
\begin{itemize}
\item if the central parameter $c=-1+ \epsilon $ and the noise amplitude  $r > 0.0399 \cdots+ |\epsilon|$, or 
\item if the central parameter $c =\tilde{c_{3}} + \epsilon $ and the noise amplitude  $r > 0.0021 + |\epsilon |.$ 
\end{itemize}
\end{cor}

\begin{proof}
By Theorem \ref{th:BififfTotDisconn}, 
$\PP_{c, r}$-almost every random Julia set is totally disconnected if $r > \rbif(c).$ 
Besides,  we have 
$\rbif(c + \epsilon) \leq \rbif(c) + |\epsilon|$ for every $\epsilon \in \CC$ 
by Theorem \ref{th:1-Lip}. 
Thus, the conclusion follows from Examples  \ref{ex:?} and \ref{ex:airplane}. 
\end{proof}

\subsection{Application to quasiconformal conjugacy}
By combining Theorem \ref{th:BififfTotDisconn} and the theory of non-autonomous holomorphic motions by Comerford, 
we can deduce that non-autonomous Julia sets are quasiconformally conjugate to the autonomous Julia set. 
In this subsection, we see an overview of this conjugacy. 
For a more detailed meaning, the reader is referred to \cite{Com08}. 

For $\omega = (c_{1}, c_{2}, \dots) \in \CC^{\NN},$ 
we denote $\sigma^{m}\omega = (c_{m+1}, c_{m+2}, \dots)$ for every $m \geq 0.$ 

\begin{definition}
Let $\delta > 0$ and let $\omega \in \CC^{\NN}$  be a bounded sequence of parameters. 
We say that $\omega$ has post-critical distance $\geq \delta$ if 
$$\inf_{m \geq 0, n\geq m} \dist (\sample{\sigma^{m} \omega}{n - m}(0), J_{\sigma^{n} \omega}) \geq \delta.$$
Here $\sample{\omega}{0}$ is understood as the identity map for every $\omega$. 
\end{definition}

The sequence $\omega$ has post-critical distance $\geq \delta$ for some $\delta > 0$ 
if and only if non-autonomous dynamics is uniformly expanding (hyperbolic)
if and only if non-autonomous dynamics is uniformly contractive at critical points. 
See \cite[Theorem 3.3]{Com08} for the detail. 

We next define hyperbolic components by analogy with the deterministic case. 
Denote $\ell^{\infty}(\CC) = \{ \omega = (c_{n})_{n=1}^{\infty} \colon \sup_{n}|c_{n}| < \infty \}.$ 

\begin{definition}
Every connected component of 
$$\{\omega \in \ell^{\infty}(\CC) \colon \omega \text{ has post-critical distance $\geq \delta$ for some $\delta > 0$}\}$$
 is called a hyperbolic component. 
\end{definition}

After formulating the concepts above, Comerford proved the following theorem \cite[Theorem 6.1]{Com08}. 

\begin{theorem}\label{th:Com}
Suppose that $\omega = (c_{n})_{n=1}^{\infty}$ and $\omega' = (c'_{n})_{n=1}^{\infty}$ lie in the same hyperbolic component. 
Then there exist $K \geq 1$ and a sequence of maps $\{\varphi_{n}\}_{n \geq 0}$
such that 
 $\varphi_{n}$ maps $J_{\sigma^{n} \omega}$ onto $J_{\sigma^{n} \omega'}$ $K$-quasiconformally and 
 $\varphi_{n+1}\circ f_{c_{n+1}}=  f_{c'_{n+1}}  \circ \varphi_{n}$ on $J_{\sigma^{n} \omega}$ for every $n \geq 0$. 
 
 We say that $\omega$ and $\omega'$ are quasiconformally conjugate on their iterated Julia sets 
 if such a sequence $\{\varphi_{n}\}_{n \geq 0}$ exists. 
\end{theorem}

More precisely, the sequence $\{\varphi_{n}\}_{n \geq 0}$ is bi-equicontinuous,  
as defined in \cite[Section 4]{Com08}. 
Using the theorem above, we show the new results. 

\begin{theorem}\label{th:qc}
Suppose that the interior of $ \B{c}{r}$ contains a superattracting parameter $\tilde{c}$ and   
suppose $r < \rbif (c)$.
Then every two $\omega, \omega' \in \B{c}{r}^{\NN}$ are quasiconformally conjugate on their iterated Julia sets. 
\end{theorem}

\begin{proof}
By Definition \ref{def:bifRad} and Theorem \ref{th:BififfTotDisconn}, 
the polynomial semigroup $G_{c, r}$ has the planar attracting minimal set $L = \overline{\bigcup_{g \in G_{c, r}}\{g(0)\}}.$
Define $\delta = \dist (L, J(G_{c,r})),$ which is strictly positive by Definition \ref{def:attractingMinSet} of attracting minimal sets. 
For every $\omega\in \B{c}{r}^{\NN}$, $m \geq 0$, and $n\geq m$, we have $\sample{\sigma^{m} \omega}{n - m}(0) \in L$ 
since $\sample{\sigma^{m} \omega}{n - m} \in G_{c, r}$. 
Also, $J_{\sigma^{m} \omega} \subset J(G_{c,r})$ by Proposition \ref{prop:J}. 
Therefore, $\omega$ has post-critical distance $\geq \delta$ with uniform $\delta > 0.$ 
Since the $\ell^{\infty}$-ball $\B{c}{r}^{\NN}$ is connected, the ball $\B{c}{r}^{\NN}$ is contained in a hyperbolic component. 
Thus,  every two $\omega, \omega' \in \B{c}{r}^{\NN}$ are quasiconformally conjugate on their iterated Julia sets by Theorem \ref{th:Com}. 
\end{proof}

More precisely, we can show by using \cite[Theorem 1.3]{Com08} that the maximal dilatation $K$ of the quasiconformally conjugacy $\{\varphi_{n}\}_{n \geq 0}$ depends on $c$ and $r$ but not on $\omega$ nor $\omega'.$ 
This is because the $\delta$ in the proof is independent of $\omega$. 
Besides, we can show that the maximal dilatation $K$ 
tends to $1$ as $r$ tends to $0$. 

Lastly, we give the following corollary. 

\begin{cor}\label{cor:qc}
Suppose that the interior of $ \B{c}{r}$ contains a superattracting parameter $\tilde{c}$. 
Denote the autonomous Julia set by $J_{\tilde{c}} = J_{\tilde{\omega}}$ where $\tilde{\omega} = (\tilde{c}, \tilde{c}, \dots ).$ 
If $r < \rbif (c)$, 
then for every  $\omega\in \B{c}{r}^{\NN}$, there exists a map $\varphi$ 
which maps $J_{\tilde{c}}$ onto the non-autonomous Julia set $J_{\omega}$ quasiconformally. 
If $r > \rbif (c),$ then for $\PP_{c, r}$-almost every $\omega$, 
the autonomous Julia set $J_{\tilde{c}}$ is not homeomorphic to the non-autonomous Julia set $J_{\omega}.$ 
\end{cor}

\begin{proof}
First suppose $r < \rbif (c)$.  
By Theorem \ref{th:qc}, there exists a quasiconformal conjugacy $\{\varphi_{n}\}_{n \geq 0}$ between $\tilde{\omega}$ and $\omega$ for every $\omega \in \B{c}{r}^{\NN}.$ 
By definition, the map $\varphi_{0}$ quasiconformally maps $J_{\tilde{c}}$ onto $J_{\omega}.$ 
Next suppose $r > \rbif (c)$.  
Then by Theorem \ref{th:BififfTotDisconn},  the non-autonomous Julia set $J_{\omega}$ is totally disconnected almost surely. 
Since the autonomous Julia set $J_{\tilde{c}}$ of superattracting parameter $\tilde{c}$ is connected, 
these $J_{\tilde{c}}$ and $J_{\omega}$ cannot be homeomorphic each other. 
\end{proof}


\section{Conclusive discussions}\label{sec:Conc}
In this section, we discuss three perspectives, namely, open problems, other problems about geometric measure theory, and application to the general settings. 

\subsection{Open problems}
We give 
the examples of $c$ satisfying equality $\rbif(c) = \dist(c, \partial \mathcal{M})$ and 
the examples of $c$ satisfying strict inequality $\rbif(c) < \dist(c, \partial \mathcal{M})$. 
The author conjectures that the former ones are uncommon. 
More precisely, the author consider the answer of the following question is yes. 

\begin{ques}
If $c$ does not belong to the main cardioid,  then does $\rbif(c) < \dist(c, \partial \mathcal{M})$ hold? 
If $c \notin [0, 1/4],$ then does it hold? 
\end{ques}

We numerically estimated  $0.0386\cdots \leq \rbif(-1) \leq  0.0399 \cdots$. 
However, the actual value remains unknown. 
More generally, for every superattracting parameter $\tilde{c}$, it is interesting to determine the bifurcation radius. 

\begin{ques}
What is the exact value of $\rbif(\tilde{c})$ for a superattracting parameter $\tilde{c}$? 
How does it relate to other dynamical quantities? 
\end{ques}


Besides, it would be meaningful to investigate their asymptotic behavior.  

\begin{ques}
Let $\tilde{c}_{p}$ be a superattracting parameter with period $p$ for every $p \in \NN.$
How does $\rbif(\tilde{c}_{p})$ decrease as $p \to \infty$? 
\end{ques}

Our strategy to estimate $\rbif(-1)$ from above is to find multiple roots of some algebraic equation. 
This raises the following question. 

\begin{ques}
For every superattracting parameter $\tilde{c}$, is the bifurcation radius $\rbif(\tilde{c})$ an algebraic number? 
\end{ques}

At least, if $\tilde{c} = 0,$ then $\rbif(0) = 1/4$ is algebraic and is related to the parabolic parameter $c =1/4$. 

We are interested also in 
parameters which are not superattracting. 
Recall that Theorem \ref{th:BififfTotDisconn} is concerned with the case where $\B{c}{r}$ contains superattracting parameters.  

\begin{ques}
Can we generalize Theorem \ref{th:BififfTotDisconn}  when $\B{c}{r}$ does not contain superattracting parameters? 
\end{ques}

\begin{ques}
How does $\rbif(c)$ decrease as $c \to \partial \Mandel$?
For example, within $\epsilon > 0$, how does $\rbif(-3/4 + \epsilon)$ decrease as $\epsilon \to 0$?
\end{ques}

\subsection{Other problems about geometric measure theory} 
One can consider the Hausdorff dimensions and other geometric properties of random Julia sets. 
For instance, the uniform perfectness and the Johnness was investigated. 
See Theorem 1.12 and Theorem 1.6 of \cite{Sumi06}. 

Regarding dimensions, the following facts are known. 
Suppose $c=0$ and $0 < r < 1/4$ and let  $(\PP_{0, r}, \Omega_{0, r})$ as in  Setting \ref{settingUniform}.
Denote by $\mathrm{HD} (J)$ the Hausdorff dimension of a set $J.$ 
Then, there exists $1 < d_{0, r} <2$ such that $\mathrm{HD}( J_{\omega}) = d_{0, r}$ for $\PP_{0, r}$-almost every $\omega \in \Omega_{0, r}.$ 
See Theorems 8.8 and 8.10 of \cite{MSU} by Mayer, Skorulski, and Urba\'nski. 
Besides, denote by $\mathcal{H}^{d}$ the $d$-dimensional Hausdorff measure and by $\mathcal{P}^{d}$ the $d$-dimensional Packing measure, 
then $\mathcal{H}^{d_{0, r}} ( J_{\omega}) = 0$ and $\mathcal{P}^{d_{0, r}} ( J_{\omega}) = \infty$ for $\PP_{0, r}$-almost every $\omega \in \Omega_{0, r}.$ 
See \cite[Theorem 8.16]{MSU}. 
Note that this is different from the deterministic case, and hence almost every $J_{\omega}$ is not bi-Lipschitz equivalent to any autonomous hyperbolic Julia sets. 

Sumi kindly taught me that for every $c \in \CC$ and $r >0$, 
there exists a constant $d_{c, r}$ such that $\mathrm{HD}( J_{\omega}) = d_{c, r}$ for $\PP_{c, r}$-almost every $\omega \in \Omega_{c, r}$. 
Rugh showed that the dimension $d_{c, r}$ depends real-analytically on 
$\mathrm{Re} (c)$,  $\mathrm{Im} (c)$, and  $r$  within $|c| + r <1/4$. 
See \cite{Rugh} for the exact statements. 
This is the motivation for the following question. 

\begin{ques}
How does the typical dimension $d_{c, r}$ depend on $c$ and $r$? 
\end{ques}

For other geometric properties, 
Br\"uck showed that the random Julia set $J_{\omega}$ is a quasi-circle for every $\omega \in \Omega_{0, r}$ if $r < 1/4.$ 
See \cite[Corollary 4.5]{B01}. 

Theorem \ref{th:qc} generalize this result. 
However, the following questions are open. 

\begin{ques}
Suppose $r > \rbif(c).$ 
Then for $\PP_{c,r} \times \PP_{c,r}$-almost every $(\omega, \omega')$, 
are the Julia sets $J_{\omega}$ and $J_{\omega'}$ quasiconformally equivalent? 
Note that almost every Julia set is homeomorphic to the (usual) Cantor set. 
\end{ques}

\begin{ques}
If $r = \rbif(c),$ 
then is almost every random Julia set $J_{\omega}$ quasiconformally equivalent to the autonomous Julia set? 
\end{ques}

In this paper, we considered the measure-theoretic aspect of the typicality problem.
For the topological aspect,  see \cite{GQL} for example.  

\subsection{Application to  general settings}
We mainly considered random dynamical systems induced by uniform noise on disks as in Setting \ref{settingUniform}. 
However,  random dynamical systems induced by any distribution $\mu$ with compact support can be considered. 
Even for Setting \ref{setting}, we can apply the numerical estimates given in Section \ref{sec:MainRes}. 
More precisely, the following holds. 

\begin{theorem}\label{th:general}
Let $\mu$ be a Borel probability measure on the parameter plane $\CC$ with compact support 
and let $(\PP_{\mu}, \Omega_{\mu})$ be defined as in Setting \ref{setting}. 
If there exists $c \in \CC$ such that $\supp \mu \subset \B{c}{\rbif(c)}$, 
then the polynomial semigroup $G_{\mu}$ has a planar minimal set. 
If there exists $c \in \CC$ such that $\mathrm{int} (\supp \mu) \supset \B{c}{\rbif(c)}$, 
then the random Julia set $J_{\omega}$ is totally disconnected for $\PP_{\mu}$-almost every $\omega \in \Omega_{\mu}$. 
\end{theorem}

\begin{proof}
Suppose that there exists $c \in \CC$ such that $\supp \mu \subset \B{c}{\rbif(c)}$. 
By definition, the polynomial semigroup $G_{c, \rbif(c)}$ has a planar minimal set. 
Thus,  Lemma \ref{lem:MinimalInclusion} gives that $G_{\mu}$ also has  a planar minimal set. 

Similarly, if there exists $c \in \CC$ such that $\mathrm{int} (\supp \mu) \supset \B{c}{\rbif(c)}$, 
then $G_{\mu}$ does not have any planar minimal sets. 
Thus, $G_{\mu}$ is mean stable, and $T_{\{\infty\}, \mu}$ is identically equal to $1$ by Theorem \ref{th:Sumi}. 
Combining this with Theorem \ref{th:T0=1ImpliesTypFastEsc}, 
we have that the random Julia set $J_{\omega}$ is totally disconnected for $\PP_{\mu}$-almost every $\omega \in \Omega_{\mu}$. 
\end{proof}

\begin{example}
Let $\mu_{s}$ be the normalized Lebesgue measure on $\CC$ with $\supp \mu_{s} =\{c \in \CC \colon |\mathrm{Re} \ c| \leq s,   |\mathrm{Im}\ c| \leq s\}$ for $s \geq 0$. 
Then there exists $s^{*} > 0$ such that $G_{\mu_{s}}$ is mean stable if and only if $s \notin \{0, s^{*}\}.$ 
Theorem \ref{th:general} and Corollary \ref{cor:rbif0} imply $(4 \sqrt2)^{-1} \leq s^{*} \leq 4^{-1}.$ 
\end{example}

Generally, we may apply our results to non-quadratic cases. 
Since ``the Mandelbrot set is universal'' (\cite{McM00}), 
the structure of the quadratic family exists for any non-trivial family. 
Thus, our quadratic results apply to random dynamical systems of any non-trivial family.  
In this case, it is difficult to obtain quantitative estimates of bifurcation parameters.


\subsection*{Acknowledgments}
The author thanks Hiroki Sumi for detailed comments and valuable discussions. 
The author thanks Isaia Nisoli for instructing the author to use validated numerics. 
The manuscript has been edited for the English language by a professional native proofreader at Xtra, Inc. 
This work was supported by the Research Institute for Mathematical Sciences,
an International Joint Usage/Research Center located in Kyoto University.
This work is supported by  
JSPS Grant-in-Aid for Research Activity Start-up Grant Number JP 21K20323 and 
Grant-in-Aid for Early-Career Scientists Grant Number JP 23K13000. 


\end{document}